\documentclass{birkjour}
\usepackage{amsmath}
\usepackage{amssymb}
\usepackage{amsfonts}
\usepackage{float}
\usepackage{graphicx}
\usepackage{cite}
\usepackage{color}
\usepackage{epstopdf}
\usepackage [latin1]{inputenc}
\usepackage{mathrsfs}
\usepackage[noend]{algpseudocode}
\usepackage{algorithmicx,algorithm}
\usepackage{lineno}
\usepackage{listings}
\usepackage{color}
\usepackage{xcolor}
\usepackage{colortbl}
\usepackage{booktabs}

\definecolor{dkgreen}{rgb}{0,0.6,0}
\definecolor{gray}{rgb}{0.5,0.5,0.5}
\definecolor{mauve}{rgb}{0.58,0,0.82}
 \newtheorem{thm}{Theorem}[section]
 
 \newtheorem{lem}{Lemma}[section]
 
 \newtheorem{rem}{Remark}
 \theoremstyle{definition}
 
 \theoremstyle{remark}

 \numberwithin{equation}{section}

\begin{document}

\title[Monochromatic vertex-disconnection of graphs]
{Monochromatic vertex-disconnection of\\ graphs}

\author[M. Fu]{Miao Fu}
\address{%
Center for Applied Mathematics \\
Tianjin University\\
Tianjin, 300354\\
P. R. China}
\email{miaofu@tju.edu.cn}

\author[Y.Q. Zhang ]{Yuqin Zhang*}
\address{%
School of Mathematics \\
Tianjin University\\
Tianjin, 300072\\
P. R. China}
\email{yuqinzhang@tju.edu.cn}

\thanks{*Yuqin Zhang: yuqinzhang@tju.edu.cn
\\School of Mathematics, Tianjin University, Tianjin, P. R. China 300072
\\Miao Fu: miaofu@tju.edu.cn
\\Center for Applied Mathematics, Tianjin University, Tianjin, P. R. China 300354
}



\date{}

\begin{abstract}
  Let $G$ be a vertex-colored graph. A vertex cut $S$ of $G$ is called a $monochromatic$ $vertex$ $cut$ if the vertices of $S$ are colored with the same color. A graph $G$ is $monochromatically$ $vertex$-$disconnected$ if any two nonadjacent vertices of $G$ has a monochromatic vertex cut separating them. The $monochromatic$ $vertex$-$disconnection$ $number$ of $G$, denoted by $mvd(G)$, is the maximum number of colors that are used to make $G$ monochromatically vertex-disconnected.
  In this paper, the connection between the graph parameters are studied: $mvd(G)$, connectivity and block decomposition.
  We determine the value of $mvd(G)$ for some well known graphs, and then characterize $G$ when $n-5\leq mvd(G)\leq n$ and all blocks of $G$ are minimally 2-connected triangle-free graphs.
  We obtain the maximum size of a graph $G$ with $mvd(G)=k$ for any $k$. Furthermore, we study the Erd\H{o}s-Gallai-type results for $mvd(G)$, and completely solve them.
  Finally, we propose an algorithm to compute $mvd(G)$ and give an $mvd$-coloring of $G$.
\end{abstract}
\maketitle

\section{ Introduction }
\quad \quad Connectivity is perhaps the most fundamental graph theoretic subject, in both a combinatorial sense and an algorithmic sense. To expand its application area, connectivity is strengthened, such as requiring graph coloring, hamiltonicity \cite{Kirkman}, conditional connectivity \cite{Harary}.

In 2008, Chartrand et al. \cite{Chartrand} introduced an interesting way, the rainbow connection, to strengthen the connectivity requirement. An edge-colored graph $G$ is rainbow connected if any two vertices are connected by a path whose edges have distinct colors. This concept comes from the communication of information between agencies of government and is also applied to communication networks \cite{Chakraborty}.
As one of the important topics in the study of colored connectivity of graphs, rainbow connection has been well studied by many scholars, and meanwhile, many colored versions of connectivity parameters have been introduced in recent years.
For example, the monochromatic connection introduced by Caro and Yuster \cite{Caro} in 2011 is defined from the monochromatic-version; the monochromatic vertex connection introduced by Cai, Li, and Wu \cite{Cai} in 2018 is defined from the vertex-version.
For more results, we refer the reader to \cite{Chandran, Dudek, GuR, Krivelevich, LuZ, LiP, MaY}.

There are two ways to study the connectivity, one using paths and the other using vertex cuts. These concepts mentioned above use paths, so it is natural to consider monochromatic vertex cuts. Let $G$ be a vertex-colored graph.
A vertex cut $S$ is called a $monochromatic$ $vertex$ $cut$ if the vertices of $S$ are colored with the same color, and a $monochromatic$ $x$-$y$ $vertex$ $cut$ is a monochromatic vertex cut that separates $x$ and $y$. Obviously, if $x$ is adjacent to $y$, there is no $x$-$y$ vertex cut, so we only need to consider nonadjacent vertices in the sequel.
Then, $G$ is called $monochromatically$ $vertex$-$disconnected$ if any two nonadjacent vertices of $G$ has a monochromatic vertex cut separating them; the corresponding coloring is called $monochromatic$ $vertex$-$disconnection$ $coloring$ ($MVD$-$coloring$ for short).
A natural question arises: how many colors are available to ensure that $G$ is monochromatically vertex-disconnected? Such an upper bound must exist and is called the $monochromatic$ $vertex$-$disconnection$ $number$ of $G$, denoted by $mvd(G)$; the corresponding coloring is called $mvd$-coloring.

In addition to being a natural combinatorial measure, our parameter can also be applied to communication networks. Suppose that $G$ represents a network (e.g., a cellular network) where messages can be transmitted between any two vertices. To intercept messages (e.g., to prevent the transmission of error messages), each vertex is equipped with an interceptor that requires a fixed password (color) to be turned on. There is a fixed interception passphrase between any two different vertices. Entering this passphrase at the vertex cut where the password matches will turn on these interceptors and intercept the message between the two vertices. To enhance system security, the number of passwords should be as large as possible. This number is precisely $mvd(G)$.

In this paper, the connection between graph parameters are studied: $mvd(G)$ and connectivity $\kappa(G)$. We obtain the following results in Section 2:
\begin{thm}\label{bound}
If $G$ is a connected non-complete graph, then $1\leq mvd(G)\leq n-\kappa^{+}(G)+1 \leq n-\kappa(G)+1$.
\end{thm}

\noindent For the graph $G$ with $\kappa(G)=2$, we give a tighter upper bound which is better than the one in Theorem \ref{bound}:
\begin{thm}\label{mini2bound}
If $G$ is a minimally $2$-connected graph of order $n\geq 4$, then $mvd(G)\leq \left\lfloor {\frac{n}{2}}\right\rfloor$.
\end{thm}

\noindent For the graph $G$ with $\kappa(G)=1$, $mvd$-coloring, a global property of $G$, is transformed into a local property of each block, which greatly simplifies the original problem:
\begin{thm}\label{blockrelation}
If $\kappa(G)=1$ and $G$ has $r$ blocks $B_{1},$ $\ldots , B_{r}$, then $mvd(G)=(\sum\limits_{i=1}^{r} mvd(B_{i}))-r+1$.
\end{thm}

We have obtained concise and powerful results when $\kappa(G)=1$. So in Section 3, we focus on the value of $mvd(G)$ for some well known graphs with $\kappa(G)\geq2$ and obtain several classes of graphs with $mvd(G)=k$, where $k\in \{1, 2, n\}$.
Moreover, we completely characterize $G$ when $n-5\leq mvd(G)\leq n$ and all blocks of $G$ are minimally 2-connected triangle-free graphs.

The Erd\H{o}s-Gallai theorem \cite{Erdos} originated in 1959 and is an interesting problem in extremal graph theory, where Erd\H{o}s-Gallai-type result aims at determining the maximum or minimum value of a graph parameter with some given properties. For a graph parameter, it is always interesting and challenging to get the Erd\H{o}s-Gallai-type results, see \cite{Allen, Faudree, LiP, Lewin, Ning} for more such results on various kinds of graph parameters. In Section 4, we study two kinds of Erd\H{o}s-Gallai-type problems for our parameter $mvd(G)$.

\textbf{Problem I}. Given two positive integers $n, k$ with $1\leq k\leq n$, compute the minimum integer $f_v(n, k)$ such that for any graph $G$ of order $n$, if $|E(G)|\geq f_v(n, k)$, then $mvd(G)\geq k$.

\textbf{Problem II}. Given two positive integers $n, k$ with $1\leq k\leq n$, compute the maximum integer $g_v(n, k)$ such that for any graph $G$ of order $n$, if $|E(G)|\leq g_v(n, k)$, then $mvd(G)\leq k$.

\noindent By Theorem \ref{blockrelation}, for any tree $G$ of order $n$ we have $mvd(G)=n$, so $g_v(n, k)$ does not exist for $1\leq k\leq n-1$. Since $mvd(K_n)=n$, where $K_n$ is a complete graph, $g_v(n, n) = \frac{n(n-1)}{2}$. The result of Problem I is shown below.
\begin{thm}\label{E-G}
Given two positive integers $n,\ k$ with $n\geq5$ and $1\leq k\leq n$,
\begin{equation}\label{E-G-equation}
    f_v(n, k)=
   \begin{cases}
   n-1,&\mbox{$k=1$,}\\
   \frac{n(n-1)}{2}-1,&\mbox{$2\leq k\leq3$,}\\
   \frac{n(n-1)}{2},&\mbox{$4\leq k\leq n$.}\nonumber
   \end{cases}
  \end{equation}
\end{thm}

\noindent Moreover, we obtain the maximum size of $G$ with $mvd(G)=k$ for any $k$.
\begin{thm}\label{size}
Given two positive integers $n,k$ with $n>4$ and $1\leq k\leq n$, the maximum size of a connected graph $G$ of order $n$ with $mvd(G)=k$ is
\begin{equation}
    |E(G)|_{max}=
   \begin{cases}
   \frac{n(n-1)}{2}-2,&\mbox{ $k=1\ and\ n\geq5$,}\\
   7,&\mbox{ $k=2\ and\ n=5$,}\\
   \frac{n(n-1)}{2}-4,&\mbox{ $k=2\ and\ n\geq 6$,}\\
   \frac{n(n-1)}{2}-k+2,&\mbox{ $3\leq k\leq n-1$,}\\
   \frac{n(n-1)}{2},&\mbox{ $k=n$.}\nonumber
   \end{cases}
\end{equation}
\end{thm}

Finally, as a convenient and applicable mathematical tool to solve the $mvd$-coloring problem, based on Theorem \ref{blockrelation}, an algorithm is proposed in Section 5 to obtain $mvd(G)$ and $mvd$ coloring of $G$, or at least to reduce the tedious work. Specifically, a type set consisting of all graphs with known $mvd$-coloring is constructed, and the algorithm performs block decomposition of the complex and huge graphs and colors the blocks that are isomorphic to the elements in the type set. The main part of the code is shown in Appendix B, and the complete code is written in Java and given on Github: https://github.com/fumiaoT/mvd-coloring.git.
\section{Some basic results}
\quad \quad All graphs considered in this paper are simple, connected, finite, undirected. We follow the terminology and notation of Bondy and Murty \cite{Bondy}. Let $n=|G|$ be the order of $G$, and let $|E(G)|$ be the size of $G$. For $D\subseteq E(G)$, $G-D$ is the graph obtained by removing $D$ from $G$. For $S\subseteq V(G)$, $G-S$ is the graph obtained by removing $S$ and the edges incident to the vertices of $S$ from $G$. We use $[r]$ to denote the set $\{1, 2, \dots,r\}$. For a vertex-coloring $\tau$ of $G$, $\tau(v)$ is the color of vertex $v$, $\tau(G)$ is the set of colors used in $G$, and $|\tau(G)|$ is the number of colors in $\tau(G)$. If $H$ is a subgraph of $G$, then the part of the coloring of $\tau$ on $H$ is called $\tau$ $restricted$ on $H$. To show the connection between $mvd(G)$ and connectivity $\kappa(G)$, we need the following lemma:
\begin{lem}\label{induced}
If $\tau$ is an $MVD$-coloring of $G$, then $\tau$ restricted on $G[S]$ is also an $MVD$-coloring of $G[S]$.
\end{lem}
\begin{proof}
Let the coloring of $\tau$ restricted on $G[S]$ be denoted as $\tau'$ and let $x,\ y$ be two nonadjacent vertices of $G[S]$. If $D$ is a monochromatic $x$-$y$ vertex cut of $G$, then $D'=D\cap V(G[S])$ is a monochromatic $x$-$y$ vertex cut in $G[S]$. Otherwise, if there is an $x,y$-path $P$ in $G[S]-D'$, then $P$ is also in $G-D$, a contradiction. Thus, $\tau'$ is an $MVD$-coloring of $G[S]$.
\end{proof}

\begin{proof}[Proof of Theorem \ref{bound}]
$G$ is a non-complete graph and $x,\ y$ are two nonadjacent vertices of $G$. Let $\kappa(x,y)$ be the minimum size of an $x$-$y$ vertex cut, and let $\kappa^{+}(G)$ be the upper bound of the function $\kappa(x,y)$. Obviously, $mvd(G)\geq 1$. For the upper bound, assume that $S$ is a monochromatic $x$-$y$ vertex cut. Therefore, $mvd(G)\leq 1+(n-|S|)\leq n-\kappa(x,y)+1$. Thus, $mvd(G)\leq n-\kappa^{+}(G)+1\leq n-\kappa(G)+1$.
\end{proof}

A graph $G$ is $minimally$ 2-$connected$ ($minimal\ block$) if $G$ is 2-connected but $G-\{e\}$ is not 2-connected for every $e\in E(G)$.
To obtain a tighter upper bound on the minimally $2$-connected graph than in Theorem \ref{bound}, we need more preparations: a nest sequence of graphs is a sequence $G_0, G_1, \ldots, G_t$ of graphs such that $G_i\subset G_{i+1}, 0\leq i<t$; an $ear\ decomposition$ of a $2$-connected graph $G$ is a nest sequence $G_0, G_1, \ldots, G_t$ of 2-connected subgraphs of $G$ satisfying the following conditions: (i) $G_0$ is a cycle of $G$, (ii) $G_{i+1}=G_{i}\cup P_{i}$, where $P_i$ is an ear of $G_i$, $0\leq i<t$, (iii) $G_{t}=G$.

\begin{lem}\cite{LiXLiuS}\label{internalofPt}
Let $G$ be a minimally $2$-connected graph, and $G$ is not a cycle. Then $G$ has an ear decomposition $G_0, G_1, \ldots, G_{t}\ (t\geq 1)$ satisfying the following conditions:

(i) $G_{i+1}=G_{i}\cup P_{i}\ (0\leq i<t)$, where $P_{i}$ is an ear of $G_{i}$ in $G$ and at least one vertex of $P_{i}$ has degree two in $G$,

(ii) each of the two internally disjoint paths in $G_0$ between the endpoints of $P_0$ has at least one vertex with degree two in $G$.
\end{lem}

\begin{lem}\cite{Plummer}\label{trianglefree}
If $G$ is a minimally $2$-connected graph of order $n\geq 4$, then $G$ contains no triangles.
\end{lem}

\begin{lem}\label{tree}
If $G$ is a cycle of order $n\geq 4$, then $mvd(G)=\left\lfloor {\frac{n}{2}}\right\rfloor$.
\end{lem}
\begin{proof}
Let $G=v_1e_1\ldots v_ne_nv_1$. Define a vertex-coloring $\tau: V(G)\rightarrow \big[\left\lfloor {\frac{n}{2}}\right\rfloor\big]$ such that if $j\equiv i$ $(mod$ $\left\lfloor {\frac{n}{2}}\right\rfloor)$, then $\tau(v_j)=i$, where $i\in\big[\left\lfloor {\frac{n}{2}}\right\rfloor\big]$ and $j\in[n]$. It can be shown that for any two nonadjacent vertices $x$ and $y$, there is a monochromatic $x$-$y$ vertex cut. So $\tau$ is an $MVD$-coloring and $mvd(G)\geq \left\lfloor {\frac{n}{2}}\right\rfloor$.

For $n\geq 4$, if $mvd(G)\geq \left\lfloor {\frac{n}{2}}\right\rfloor+1$ and $\tau$ is an $MVD$-coloring of $G$ with $|\tau(G)|\geq \left\lfloor {\frac{n}{2}}\right\rfloor+1$, then there is a color, say $i$, that colors only one vertex $v_j$. Otherwise, $V(G)\geq 2|\tau(G)| \geq n+1$. Since $\kappa(G)=2$, the monochromatic $v_{j-1}$-$v_{j+1}$ vertex cut must contain $v_j$ and some vertex in $G-\{v_{j-1}, v_j, v_{j+1}\}$, which contradicts the fact that $\tau$ is an $MVD$-coloring.
\end{proof}

\begin{proof}[Proof of Theorem \ref{mini2bound}]We have the following claim:

\textbf{Claim 1}: If $G$ is a 2-connected triangle-free graph and every ear $P_{i}$ $(0\leq i< t)$ has internal vertices, then $mvd(G)\leq \left\lfloor {\frac{n}{2}}\right\rfloor$.

Let $F=\{G_0, G_1, \ldots, G_t\}$ be an ear decomposition of $G$. We use induction on $|F|$. By Lemma \ref{tree}, the theorem holds for $|F|=1$.
If $|F|=t+1>1$, let $\tau$ be an $mvd$-coloring of $G$. Since $|P_{t-1}|\geq 3$, $G_{t-1}$ is a connected vertex induced subgraph of $G$. By Lemma \ref{induced}, $\tau$ restricted on $G_{t-1}$ is an $MVD$-coloring of $G_{t-1}$. By induction, we have
$$|\tau(G_{t-1})|\leq mvd(G_{t-1})\leq \left\lfloor {\frac{|G_{t-1}|}{2}}\right\rfloor=\left\lfloor {\frac{n-|P_{t-1}|+2}{2}}\right\rfloor.$$
Suppose that the endpoints of $P_{t-1}$ are $a$, $b$, and $L$ is the shortest $a,b$-path in $G_{t-1}$. Since $P_{t-1}$ is the last ear, cycle $C=L\cup P_{t-1}$ is a connected vertex induced subgraph of $G$ and $\tau$ restricted on $C$ is an $MVD$-coloring of $C$. Then there are at most $|P_{t-1}|-2$ vertices are assigned colors in $\tau(G)-\tau(G_{t-1})$. Since $|C|\geq 4$, each color in $\tau(G)-\tau(G_{t-1})$ colors at least two internal vertices of $P_{t-1}$. Otherwise, if $j\in \tau(G)-\tau(G_{t-1})$ and only colors one internal vertex of $P_{t-1}$, say $x_{j}$, then $x_{j-1}, x_{j+1}$ are two nonadjacent vertices of $G$ and the monochromatic $x_{j-1}$-$x_{j+1}$ vertex cut must contains $x_{j}$ and another vertex, a contradiction. Then $|\tau(G)-\tau(G_{t-1})|\leq \left\lfloor {\frac{|P_{t-1}|-2}{2}}\right\rfloor$. So,
\begin{equation}
\begin{aligned}
mvd(G)&=|\tau(G)|=|\tau(G_{t-1})|+|\tau(G)-\tau(G_{t-1})|\\
&\leq \left\lfloor {\frac{n-|P_{t-1}|+2}{2}}\right\rfloor+\left\lfloor {\frac{|P_{t-1}|-2}{2}}\right\rfloor \leq\left\lfloor {\frac{n}{2}}\right\rfloor.\nonumber
\end{aligned}
\end{equation}
Above all,  $mvd(G)\leq \left\lfloor {\frac{n}{2}}\right\rfloor$ and the bound is sharp.

By Lemma \ref{internalofPt} and \ref{trianglefree}, if $G$ is not a cycle, then there is an ear decomposition satisfying the conditions in Claim 1. Thus, $mvd(G)\leq \left\lfloor {\frac{n}{2}}\right\rfloor$. If $G$ is a cycle, then by Lemma \ref{tree}, $mvd(G)=\left\lfloor {\frac{n}{2}}\right\rfloor$ and the bound is sharp.
\end{proof}

A $block$ is a maximal connected subgraph of $G$ that has no cut-vertex. Every block of a nontrivial connected graph is either $K_2$ or a 2-connected subgraph, called $trivial$ and $nontrivial$, respectively. To show the connection between $mvd(G)$ and block decomposition, we
need the following result:
\begin{thm}\label{everyblock}
Let $G$ be a connected graph with $r$ blocks. $\tau$ is an $mvd$-coloring of $G$ if and only if $\tau$ restricted on each block is an $mvd$-coloring of each block and the colors of different blocks are different except at the cut-vertices.
\end{thm}
\begin{proof}
Let \{$B_1, \ldots, B_r$\} be the block decomposition of $G$. $\tau$ is a vertex-coloring of $G$ and $\tau_i$ is the coloring of $\tau$ restricted on $B_i$, $i\in [r]$. The theorem holds for $G$ without cut-vertices. Now let $G$ have at least one cut-vertex.

\textbf{Claim 1}: $\tau$ is an $MVD$-coloring of $G$ if and only if $\tau$ restricted on each block is an $MVD$-coloring of each block.

Since each block is a vertex induced subgraph of $G$, the necessity is obvious by Lemma \ref{induced}. Now let $\tau_i$ be an $MVD$-coloring of $B_i$, where $i\in [r]$. For any two nonadjacent vertices $x$ and $y$ in $G$, if there is a block, say $B_1$, which contains both $x$ and $y$, then any monochromatic $x$-$y$ vertex cut in $B_1$ is also a monochromatic $x$-$y$ vertex cut in $G$. Otherwise, there is an $x, y$-path $P\not\subset B_1$ and since $B_1$ is a maximal 2-connected graph, $P$ must be contained in some cycles of $G$. Then $P\subseteq B_1$, a contradiction. If $x$ and $y$ are in different blocks, then there is exactly one internally disjoint $x, y$-path containing at least one cut-vertex $v$. The vertex $v$ is a monochromatic $x$-$y$ vertex cut in $G$.

Next, let $\tau_i$ be an $mvd$-coloring of $B_i$ satisfying that if $B_i\cap B_j=v$, then $\tau_i(B_i)\cap \tau_j(B_j)=\tau_i(v)=\tau_j(v)$ and if $B_i\cap B_j=\varnothing$, then $\tau_i(B_i)\cap \tau_j(B_j)=\varnothing$, where $i, j\in[r]$ and $i\neq j$. Then $\tau$ is an $MVD$-coloring of $G$ by Claim 1. We claim $\tau$ is an $mvd$-coloring of $G$. Otherwise, there is an $mvd$-coloring $\tau'$ of $G$ satisfying $|\tau'(G)|>|\tau(G)|$. Let the coloring of $\tau'$ restricted on $B_i$ be $\tau'_i$, which is an $MVD$-coloring of $B_{i}$ by Claim 1. Then for any $i\in [r]$, $|\tau'_i(B_i)|\leq |\tau_i(B_i)|$, which contradicts $|\tau'(G)|>|\tau(G)|$.

Now to prove the necessity of the theorem. Let $\tau$ be an $mvd$-coloring of $G$. Then $\tau_i$ is an $MVD$-coloring of $B_i$ by Claim 1, where $i\in[r]$. It can be shown that if $B_i\cap B_j=v$, then $\tau_i(B_i)\cap \tau_j(B_j)=\tau_i(v)=\tau_j(v)$ and if $B_i\cap B_j=\varnothing$, then $\tau_i(B_i)\cap \tau_j(B_j)=\varnothing$, where $i, j\in[r]$ and $i\neq j$. We only need to prove that for any $i\in[r]$, $\tau_{i}$ is an $mvd$-coloring of $B_{i}$. Assuming that $\tau_1$ is not an $mvd$-coloring of $B_1$, let $\tau'_1$ be an $mvd$-coloring of $B_1$ satisfying that all colors in $\tau'_1(B_1)$ are unused colors, except for those owned by cut-vertices. Then $|\tau'_1(B_1)|>|\tau_1(B_1)|$ and we get a vertex-coloring $\tau'$ of $G$ with$|\tau'(G)|>|\tau(G)|$. According to the sufficiency of the theorem, $\tau'$ is also an $mvd$-coloring of $G$, which contradicts the maximality of $\tau$.
\end{proof}

\begin{proof}[Proof of Theorem \ref{blockrelation}]
Let $G$ be a connected graph with blocks $B_1, B_2, \ldots , B_r$, and let $\tau$ be an $mvd$-coloring of $G$. We use induction on $r$. The result holds for $r=1$. If $r>1$, then when $G$ is not 2-connected, we know that there is a block, say $B_r$, containing only one cut-vertex, say $v$. Let $G'=G-(V(B_r)-\{v\})$. Then $G'$ is a connected graph with blocks $B_1, B_2,\ldots , B_{r-1}$. By Theorem \ref{everyblock}, $\tau$ restricted on $G'$ is an $mvd$-coloring of $G'$, and combined with induction hypothesis, we have $|\tau(G')|=mvd(G')=(\sum\limits_{i=1}^{r-1} mvd(B_i))-(r-1)+1$.
According to Theorem \ref{everyblock}, $\tau$ restricted on $B_r$ is an $mvd$-coloring of $B_r$ and $\tau(B_r)\cap \tau(G')=\tau(v)$. So we deleted $|mvd(B_r)|-1$ colors from $\tau(G)$ to obtain $\tau(G')$, $mvd(G)$ is as desired.
\end{proof}

\section{ Results for special graphs }
\quad \quad If $G$ and $H$ are vertex-disjoint, then let $G\vee H$ denote the $join$ of $G$ and $H$, which is obtained from $G$ and $H$ by adding edges $\{xy:x\in V(G), y\in V(H)\}$. If $C_{n-1}$ is a cycle of order $n-1$, then $W_{n}=C_{n-1}\vee K_1$ is called a $wheel\ graph$. We first show several classes of graphs with $mvd(G)=k$, where $k\in \{1, 2, n\}$.
\begin{thm}\label{mvd=1}
If $G$ is one of the following graphs, then $mvd(G)=1$.

(i) $G$ is a wheel graph other than $W_4$;

(ii) $G=K_{n_1, \ldots, n_k}$ is a complete $k$-partite graph with $n_k$, $n_{k-1}\geq 2$ and $k>2$.
\end{thm}
\begin{proof}
\textbf{(1)} Let $G=W_n=C_{n-1}\vee K_1$, where cycle $C_{n-1}=v_1v_2\ldots v_{n-1}v_1$. It is known that $mvd(W_4)=mvd(K_4)=4$. We claim $mvd(W_n)=1$ for $n>4$.

Let $\tau$ be an $mvd$-coloring of $W_{n}$, say $\tau(K_1)=1$. Since $n>4$, $v_1$ and $v_3$ are two nonadjacent vertices with three internally disjoint $v_1,v_3$-paths, namely $v_1v_2v_3$, $v_1K_1v_3$ and $v_1v_{n-1}v_{n-2}\cdots v_4v_3$. Since $\kappa(W_n)=3$, any monochromatic $v_1$-$v_3$ vertex cut must contain vertex set $\{v_2, K_1, v_i\}$, where $v_i\in\{v_{n-1}, v_{n-2},\ldots, v_{4}\}$, so $\tau(v_2)=1$. Similarly, $v_2$ and $v_4$ are two nonadjacent vertices and we get $\tau(v_3)=1$. Repeat operations above till all vertices of $C_{n-1}$ are colored, and we get $\tau(K_1)=\tau(v_1)=\tau(v_2)=\cdots=\tau(v_{n-1})=1$. Therefore, $mvd(W_n)=1$ for $n>4$.

\textbf{(2)} Let $G=K_{n_1, n_2, \ldots, n_k}$ be a complete $k$-partite graph of order $n$, where $k\geq2$ and $1\leq n_1\leq n_2\leq\cdots\leq n_k$. Let $V_1, V_2, \ldots, V_k$ be the vertex-partition sets of $G$ with $|V_i|=n_i$, where $i\in[k]$. There are four cases below.\\
\textbf{Case 1.} $n_i=1$ for $i\in[k]$, then $mvd(G)=mvd(K_n)=n$.\\
\textbf{Case 2.} $n_i=1$ for $i\in[k-1]$, and $n_k\geq2$

Define a vertex-coloring $\tau:V(G)\rightarrow[n-k+2]$ such that $\tau(V_k)\rightarrow[n-k+1]$ and $\tau(V_i)=n-k+2$ for $i\in [k-1]$. If $x$ and $y$ are two nonadjacent vertices in $G$, then $x,\ y\in V_k$, and $V(G)-V_k$ is a monochromatic $x$-$y$ vertex cut in $G$. Thus, $\tau$ is an $MVD$-coloring of $G$ and $mvd(G)\geq n-k+2$. On the other hand, since any two vertices in $V_k$ have $k-1$ internally disjoint paths, according to Theorem \ref{bound}, $mvd(G)\leq n-\kappa^{+}(G)+1= n-k+2$.\\
\textbf{Case 3.} $k>2$ and $n_k\geq n_{k-1}\geq2$

$x$ and $y$ are nonadjacent. If $x,\ y\in V_{k-1}$, then any $x$-$y$ vertex cut must contain $V(G)-V_{k-1}$. So $V(G)-V_{k-1}$ are assigned the same color. Similarly, if $x,\  y\in V_k$, then $V(G)-V_k$ are assigned the same color. Since $k>2$, the sets $V(G)-V_{k-1}$ and $V(G)-V_{k-1}$ intersect. Then $mvd(G)=1$.\\
\textbf{Case 4.} $k=2$, $n_2\geq n_1\geq 2$

Similarly, since $k=2$, the sets $V(G)-V_1$ and $V(G)-V_2$ are disjoint, then $mvd(K_{n_1, n_2})=2$.
\end{proof}

\begin{figure}
\centering
  \includegraphics[height=2.7cm]{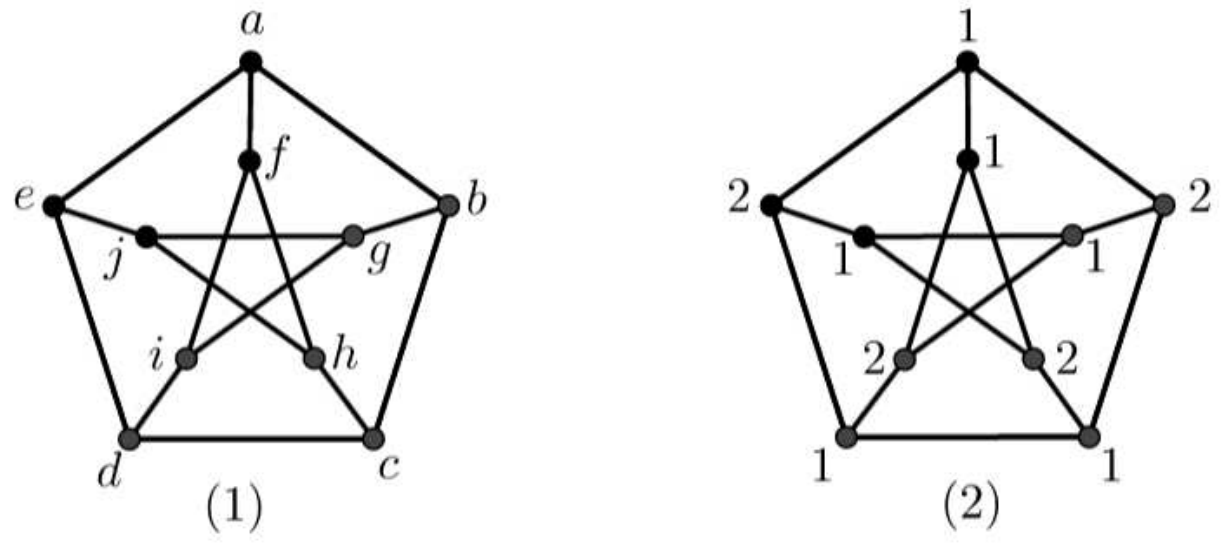}
  \caption{Vertex-coloring of Peterson graph}
  \label{2}
\end{figure}

The $cartesian\ product$ of $G$ and $H$, written as $G\Box H$, is the graph with vertex set $V(G)\times V (H)$ specified by putting $(u, v)$ adjacent to $(u', v')$ if and only if either $u=u'$ and $vv'\in E(H)$, or $v=v'$ and $uu'\in E(G)$. If $P_n$ is a path with order $n$, then $P_{m}\Box P_{n}$ is called the $m$-by-$n$ $grid$.
\begin{thm}\label{mvd=2}
If $G$ is one of the following graphs, then $mvd(G)=2$.

(i) $G=P_m\Box P_n$ is a nontrivial grid other than $P_1\Box P_n$ with $n>2$;

(ii) $G$ is a Petersen graph.
\end{thm}
\begin{proof}
\textbf{(1)} Let $G=P_m\Box P_n$ and define $x_{i,j}$ to be the vertex in the $i$-th row and $j$-th column, where $i\in [m]$ and $j\in [n]$. It is known that $mvd(P_1\Box P_n)=mvd(P_n)=n$. Then $mvd(P_1\Box P_2)=2$ and $mvd(P_1\Box P_n)>2$ for $n>2$. We claim that $mvd(P_m\Box P_n)=2$ for $m, n\geq2$.

Define a vertex-coloring $\tau$ of $G$: $V(G)\rightarrow [2]$ such that $\tau(x_{i,j})=1$ if $i+j$ is even and $\tau(x_{i,j})=2$ if $i+j$ is odd. For any vertex $x$ in $G$, the set $N(x)$ is monochromatic. Thus, for any two nonadjacent vertices $x$ and $y$ in $G$, $N(x)$ is a monochromatic $x$-$y$ vertex cut. So, $mvd(G)\geq 2$.

Now we prove $mvd(G)=2$. Any $MVD$-coloring of a 4-cycle can have only two cases, one is trivial and the other is to assign colors 1, 2 to the four vertices of the 4-cycle alternately. Suppose that $mvd(G)>2$ and $\tau$ is an $mvd$-coloring of $G$. By Lemma \ref{induced}, $\tau$ restricted on each $4$-cycle $G[x_{i,j}, x_{i,j+1},$ $x_{i+1,j+1}, x_{i+1,j}]$ is an $MVD$-coloring, $1\leq i<m, 1\leq j<n$, which contradicts that $mvd(G)>2$. Therefore, $mvd(P_m\Box P_n)=2$ for $m, n\geq 2$.

\textbf{(2)} Define a vertex-coloring $\tau$ of $G$: $V(G)\rightarrow [2]$ as shown in Fig. \ref{2}$(2)$. For any two nonadjacent vertices $x$ and $y$, there is only one common neighbor, say $z$. Suppose $\tau(z)=t$, it can be shown that the set of all vertices colored by $t$ except $x$ and $y$ is a monochromatic $x$-$y$ vertex cut. Thus, $\tau$ is an $MVD$-coloring of $G$ and $mvd(G)\geq 2$. We prove that $mvd(G)=2$ below.

Any $MVD$-coloring of $C_5$ can only be two cases, one is trivial and the other is to assign colors 1, 2 to the five vertices of $C_5$ alternately. At least two adjacent vertices in $C_5$ have the same color. Suppose that $mvd(G)>2$ and $\tau$ is an $mvd$-coloring of $G$. Let the four 5-cycles of $G$ be $G_1=G[a,b,c,d,e], G_2=G[c,d,i,f,h], G_3=G[a,b,c,h,f]$ and $G_4=G[f,h,j,g,i]$. Since $\tau$ restricted on $G_1$ is an $MVD$-coloring, there are two cases.\\
\textbf{Case 1}. $G_1$ is colored nontrivially.

Suppose that $\tau(a)=\tau(c)=\tau(d)=1, \tau(b)=\tau(e)=2$. $\tau$ restricted on $G_2$ is an $MVD$-coloring. If $G_2$ is colored trivially, i.e., $\tau(f)=\tau(h)=\tau(i)=1$, it is obvious that $\tau$ is not an $MVD$-coloring restricted on $G_3$, which contradicts that $\tau$ is an $MVD$-coloring of $G$. If $G_2$ is colored nontrivially, i.e., $\tau(f)=1$ and $\tau(h)=\tau(i)=2$, then $\tau$ is a nontrivial $MVD$-coloring restricted on $G_4$ with $\tau(g)=\tau(j)=1$, which contradicts that $mvd(G)>2$.\\
\textbf{Case 2}. $G_1$ is colored trivially.

Suppose that $\tau(a)=\tau(b)=\tau(c)=\tau(d)=\tau(e)=1$. Then $\tau$ is a trivial $MVD$-coloring restricted on $G_3$ with $\tau(f)=\tau(h)=1$. Since $\tau$ is an $MVD$-coloring restricted on $G_4$, $|\tau(G_4)|\leq 2$, which contradicts that $mvd(G)>2$.

Above all, $mvd(G)=2$.
\end{proof}

\begin{thm}\label{mvd=n}
Let $G$ be a connected graph of order $n$. Then $mvd(G)=n$ if and only if each block of $G$ is complete.
\end{thm}
\begin{proof}
Let \{$B_1, B_2, \ldots , B_r$\} be a block decomposition of $G$. If $B_i\ (i\in [r])$ is complete, we define a coloring $\tau: V(G)\rightarrow [n]$ such that all vertices of $G$ have different colors. By Lemma \ref{blockrelation}, $\tau$ is an $mvd$-coloring of $G$, and $mvd(G)=n$. On the contrary, if $mvd(G)=n$, we define a coloring $\tau: V(G)\rightarrow [n]$ such that all vertices of $G$ have different colors. By Lemma \ref{blockrelation}, $\tau(B_i)$ is an $mvd$-coloring of $B_i$. Then $B_i$ is complete. Otherwise, since $B_i$ is $2$-connected, by Theorem \ref{bound}, $mvd(B_i)\leq |B_i|-\kappa(B_i)+1\leq |B_i|-1$, a contradiction.
\end{proof}

\begin{table}[h]
\caption{$mvd(G)$ for minimally 2-connected graph $G$}\label{6-10}
\scriptsize
\begin{tabular}{p{0.8cm}|p{0.9cm}|p{1.2cm}|p{1.2cm}|p{1.2cm}|p{1.6cm}|p{2.5cm}}
\hline
\hline
$mvd(G)$  & $n\leq5$ & $n$=6 & $n$=7 & $n$=8 & $n$=9& $n$=10 \\
\hline
5 & - & - & - & - & - & $P(4,4)$ \\
\hline
4 & - & - & - & $P(3,3)$ & $P(4,3)$ $P(5,1,1)$ $P(3,3,1)$ & \textcircled{\tiny{48}}-\textcircled{\tiny{50}} $P(5,1,1,1)$ $P(5,2,1)$ $P(4,2,2)$ $P(4,3,1)$ $P(3,3,2)$ $P(6,1,1)$ $P(3,3,1,1)$ \\
\hline
3 & $K_3$ & $P(2,2)$ & $P(3,2)$ $P(3,1,1)$ & \textcircled{\tiny{4}} $P(3,2,1)$ $P(2,2,2)$ $P(4,1,1)$ $P(3,1,1,1)$ & \textcircled{\tiny{12}}-\textcircled{\tiny{16}} $P(3,2,1,1)$ $P(4,2,1)$ $P(3,4*1)$ $P(3,2,2)$ $P(4,1,1,1)$ & \textcircled{\tiny{28}}-\textcircled{\tiny{47}} $P(4,2,1,1)$ $P(3,2,2,1)$ $P(4*2)$ $P(4,4*1)$ $P(3,5*1)$ $P(3,2,3*1)$\\
\hline
2 & $C_4$ $K_2$ $P(2,1)$ $P(1,1,1)$  & $P(2,1,1)$ $P(1,1,1,1)$ & \textcircled{\tiny{1}} $P$(5*1) $P(2,2,1)$ $P(2,1,1,1)$ & \textcircled{\tiny{1}}-\textcircled{\tiny{3}} $P(2,4*1)$ $P(2,2,1,1)$ $P(6*1)$ & \textcircled{\tiny{1}}-\textcircled{\tiny{11}} $P$(7*1) $P(2,5*1)$ $P(2,2,2,1)$ $P(2,2,1,1,1)$ & \textcircled{\tiny{1}}-\textcircled{\tiny{27}} $P(2,2,4*1)$ $P(8*1)$ $P(2,6*1)$ $P(2,2,2,1,1)$\\
\hline
\hline
\end{tabular}
\end{table}

Now, we focus on minimally $2$-connected graphs \cite{Hobbs} of order 10 or less. As an example, see Fig. \ref{3}, let the three 4-cycles of $G$ be $G_1=G[a,b,c,d], G_2=G[a,b,c,e],$ $G_3=G[a,f,c,d]$. According to Lemma \ref{induced}, if $\tau$ is an $mvd$-coloring of $G$, then $\tau(G_1)$ may be (1) or (3), i.e., an $MVD$-coloring of $G_1$. For case (1), $\tau(G_2)$ and $\tau(G_3)$ must be (2). For case (3), $\tau(G_2)$ and $\tau(G_3)$ must be (4). It can be shown that both (2) and (4) are $MVD$-coloring of $G$ and (4) is an $mvd$-coloring of $G$.
Using the same method, after tedious calculations, the $mvd(G)$ of minimally 2-connected graph $G$ is shown in Table \ref{6-10}, and the $mvd$-coloring of $G$ is shown in Appendix A
(In fact, this method is applicable to many graphs with small order).

\begin{figure}[h]
\centering
  \includegraphics[height=1.6cm]{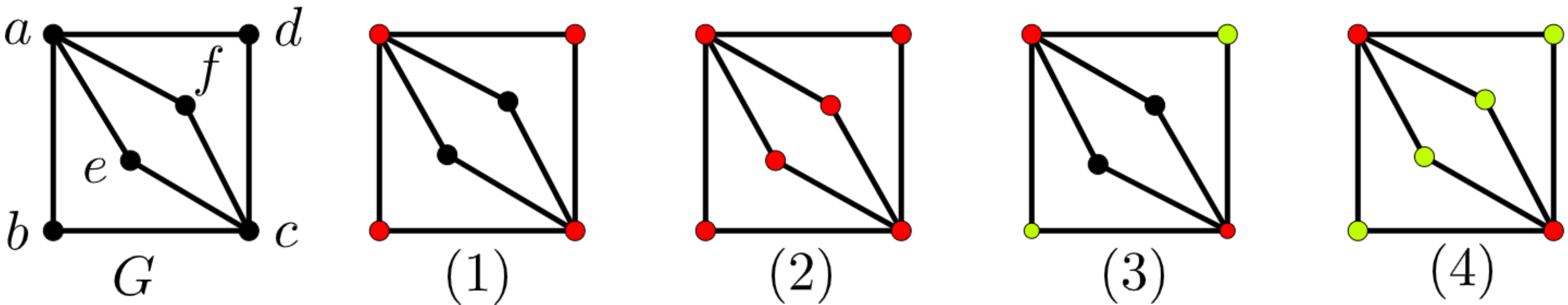}
  \caption{An example}
  \label{3}
\end{figure}

\begin{rem}
Let $P_{n_1}, \ldots, P_{n_k}$ be $k$ disjoint paths with $|P_{n_i}|=n_i$, and let $f(P_{n_i})$ and $l(P_{n_i})$ denote the first and the last vertices of $p_{i}$, respectively. Let $P(n_1, \ldots, n_k)$ be the graph with vertex set $\{\cup_{i\in[k]} V(P_{n_i})\} \cup \{u,\ v\}$ and edge set $\cup_{i\in[k]} \big[E(P_{n_i}) \cup \{f(P_{n_i})u, l(P_{n_i})v\}\big]$, where vertices $u, v\notin \cup_{i\in[k]}$ $V(P_{n_i})$. If $n_{i+1}=\cdots=n_{i+j}$, then we can write $P(n_1,\ldots , n_k)$ in the form $P(n_1,\ldots , n_i, j*n_{i+1}, n_{i+j+1}, \ldots ,n_k)$.
\end{rem}

Finally, when $n-5\leq mvd(G)\leq n$ and all blocks of the graph $G$ are minimally 2-connected triangle-free graphs, we characterize $G$. We need the following lemma:
\begin{lem}\label{blockbound}
$G$ is a connected graph of order $n$ with $r$ blocks, where $t$ blocks are trivial. If all blocks are minimally $2$-connected triangle-free graphs, then $mvd(G)\leq \left\lfloor {\frac{n+2t-r+1}{2}}\right\rfloor$.
\end{lem}
\begin{proof}
We claim that a connected graph $G$ with blocks $B_1, B_2, \ldots , B_r$ has $(\sum\limits_{i=1}^{r} |B_i|)-r+1$ vertices. The proof proceeds by induction on $r$. The result holds for $r=1$. If $r>1$, then when $G$ is not $2$-connected, we know that there is a block, say $B_r$, containing only one cut-vertex, say $v$. Let $G'=G-(V(B_r)-\{v\})$. Then $G'$ is a connected graph with blocks $B_1, B_2,\ldots , B_{r-1}$. By the induction hypothesis, $|G'|=(\sum\limits_{i=1}^{r-1} |B_i|)-(r-1)+1$. Since we deleted $|B_r|-1$ vertices from $G$ to obtain $G'$, the number of vertices in $G$ is as desired.

Without loss of generality, let the trivial blocks be $B_1, \ldots , B_t$, and let the nontrivial blocks be $B_{t+1}, \ldots , B_{r}$. By Theorem \ref{mini2bound} and \ref{blockrelation}, we have
\begin{equation}
\begin{aligned}
mvd(G)&=(\sum\limits_{i=1}^{r} mvd(B_{i}))-r+1\leq 2t+\left\lfloor {\frac{|B_{t+1}|}{2}}\right\rfloor+\ldots +\left\lfloor {\frac{|B_{r}|}{2}}\right\rfloor-r+1\\
&\leq \left\lfloor {\frac{2t+|B_{t+1}|+\ldots +|B_{r}|}{2}}\right\rfloor+t-r+1=\left\lfloor {\frac{n+r-1}{2}}\right\rfloor+t-r+1\\
&=\left\lfloor {\frac{n+2t-r+1}{2}}\right\rfloor.\nonumber
\end{aligned}
\end{equation}
\end{proof}

\begin{figure}
\centering
  \includegraphics[height=8.7cm]{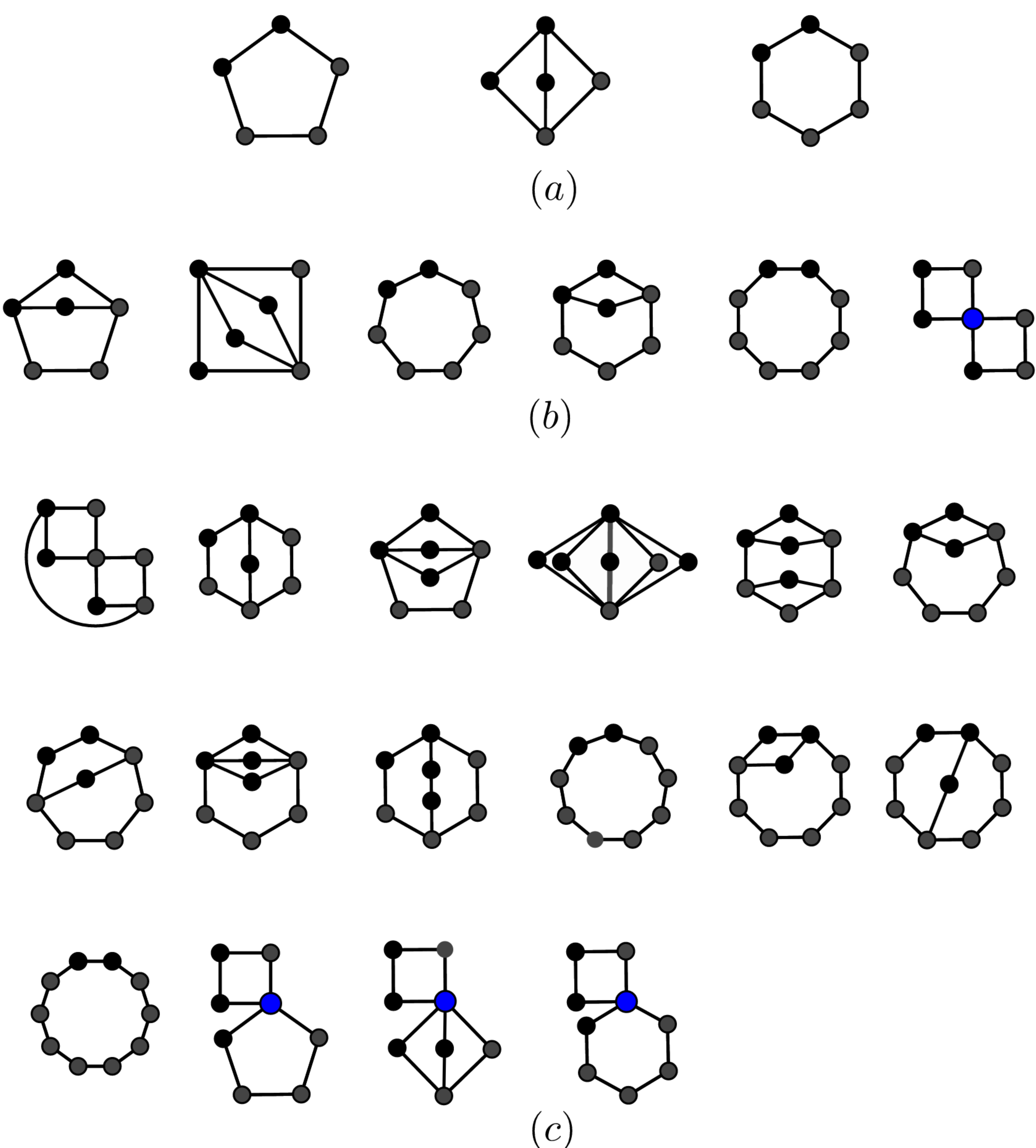}
  \caption{Results for Theorem \ref{large}}
  \label{4}
\end{figure}
Suppose that the subgraph induced by all nontrivial blocks of $G$ is isomorphic to exactly one of the graphs in Fig. \ref{4}$(a)$, then we define $\mathscr{A}=\{G\}$. Similarly, we define $\mathscr{B}$ and $\mathscr{C}$ according to Fig. \ref{4}$(b)$ and $(c)$. The blue vertex indicates that two blocks may or may not be adjacent to each other.

\begin{thm}\label{large}
For a connected graph $G$, if blocks in $G$ are all minimally $2$-connected triangle-free graphs, then
\begin{equation}
    mvd(G)=
   \begin{cases}
   n,&\mbox{ $\Leftrightarrow G$ is a tree,}\\
   n-1,&\mbox{ $\Leftrightarrow G=\varnothing$,}\\
   n-2,&\mbox{ $\Leftrightarrow G$ is a unicycle graph with cycle $C_4$,}\\
   n-3,&\mbox{ $\Leftrightarrow G\in \mathscr{A}$,}\\
   n-4,&\mbox{ $\Leftrightarrow G\in \mathscr{B}$,}\\
   n-5,&\mbox{ $\Leftrightarrow G\in \mathscr{C}$.}\nonumber
   \end{cases}
  \end{equation}
\end{thm}
\begin{proof}
By Table \ref{6-10}, Theorem \ref{blockrelation} and \ref{everyblock}, it is easy to verify the sufficiency. Now we prove the necessity. If $mvd(G)=n$, by Theorem \ref{mvd=n}, $B_i$ is complete. Since $G$ is triangle-free, $G$ is a tree.

Now suppose $mvd(G)\geq n-5$ and $G$ has $r$ blocks, of which $t$ are trivial. By Lemma \ref{blockbound}, $n-5\leq mvd(G)\leq \left\lfloor {\frac{n+2t-r+1}{2}}\right\rfloor$. There are two cases below.\\
\textbf{Case 1.} $n-r$ is even.

So $2t\geq n+r-10$. Since $t\leq r$, $r\geq n-10$, then $r$ may be $n-2,\ n-4,\ n-6,\ n-8$ or $n-10$.

When $r=n-2$, combined with $2t\geq n+r-10$ and the fact that $r=t$ if and only if $G$ is a tree (i.e., $r=n-1$), we have $n-6\leq t\leq n-3$. Since $G$ is triangle-free, when $t=n-3, n-4, n-5$ or $n-6$, respectively, we have $\sum\limits_{i=1}^{r} |B_{i}|\geq 4+2(n-3), 8+2(n-4), 12+2(n-5)$ or $16+2(n-6)$, contradicting the fact that $n=(\sum\limits_{i=1}^{r} |B_{i}|)-r+1$. Note that in other cases, we first use this method to determine the number of nontrivial blocks in $G$.

Similarly, when $r=n-4$, $n-7\leq t\leq n-5$ and there is only one nontrivial block in $G$. Since $n=(\sum\limits_{i=1}^{r} |B_{i}|)-r+1$, this nontrivial block is of order $5$, i.e., $P(2,1)$ or $P(1,1,1)$. According to Table \ref{6-10}, Theorem \ref{blockrelation} and \ref{everyblock}, $mvd(G)=n-3$ in both cases.

When $r=n-6$, $n-8\leq t\leq n-7$. Since $n=(\sum\limits_{i=1}^{r} |B_{i}|)-r+1$, if $t=n-7$, then the order of this nontrivial block is $7$; if $t=n-8$, then there are two nontrivial blocks $B_{i}$ and $B_{j}$ with $|B_{i}|+|B_{j}|=9$. Since $G$ is triangle-free, the subgraph induced by all nontrivial blocks of $G$ is one of the graphs in Fig. \ref{8}. According to Table \ref{6-10}, Theorem \ref{blockrelation} and \ref{everyblock}, $mvd(G)=n-4$ when the subgraph induced by all nontrivial blocks of $G$ is $P(3,2)$ or $P(3,1,1)$, and $mvd(G)=n-5$ for the rest cases.

\begin{figure}
\centering
  \includegraphics[height=8.2cm]{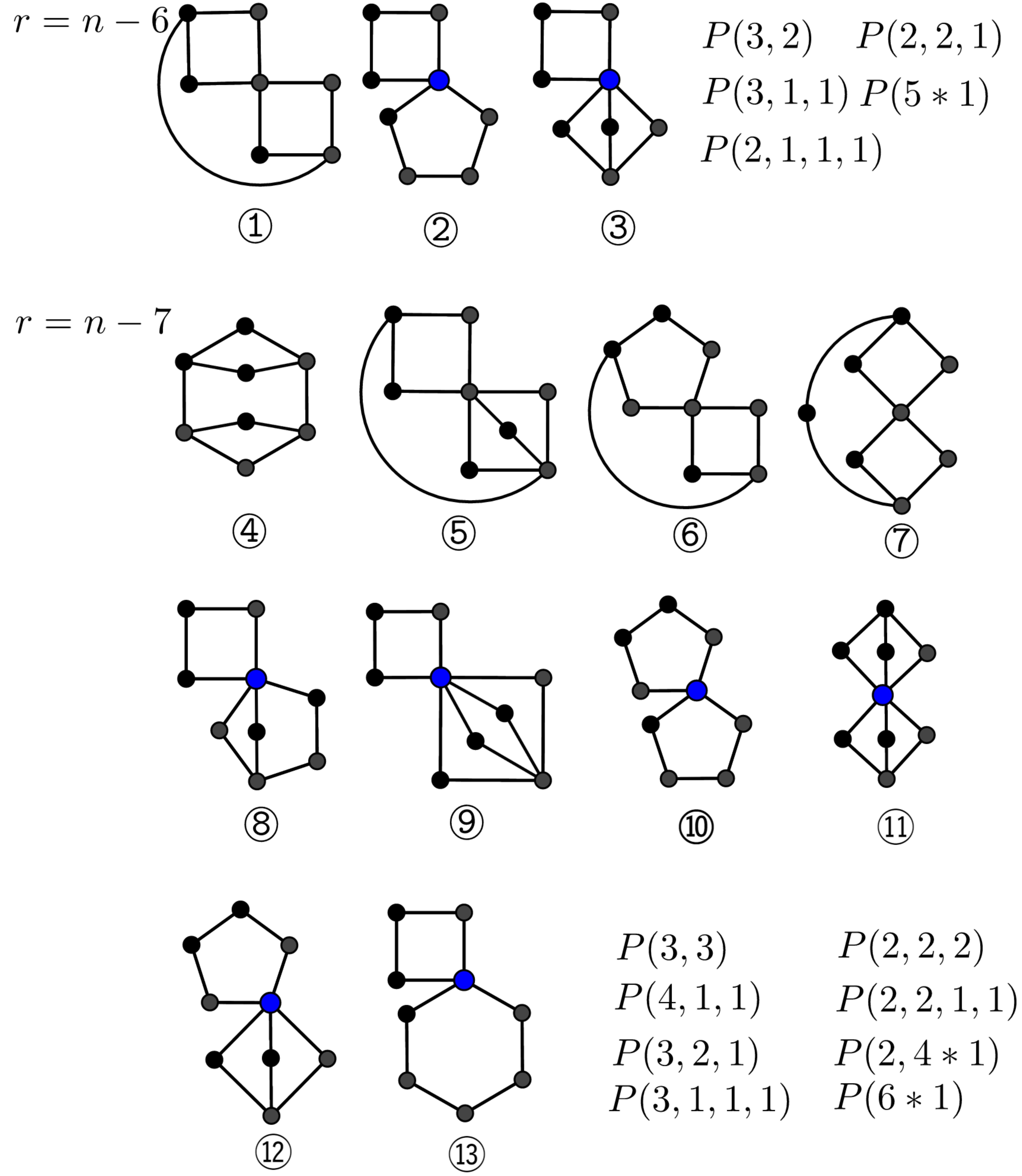}
  \caption{Proof process of Theorem \ref{large}}
  \label{8}
\end{figure}

When $r=n-8$, $t=n-9$. Since $n=(\sum\limits_{i=1}^{r} |B_{i}|)-r+1$, this nontrivial block is of order $9$, i.e., one of the graphs in Appendix A. $9$ VERTEX. According to Table \ref{6-10}, Theorem \ref{blockrelation} and \ref{everyblock}, $mvd(G)=n-5$ when the subgraph induced by all nontrivial blocks of $G$ is one of $\{P(4,3), P(5,1,1), P(3,3,1)\}$, and $mvd(G)<n-5$ for the rest cases.

When $r=n-10$, $n-10\leq t\leq n-11$, a contradiction.\\
\textbf{Case 2.} $n-r$ is odd.

So $2t\geq n+r-11$. Since $t\leq r$, $r\geq n-11$, then $r$ may be $n-1,\ n-3,\ n-5,\ n-7,\ n-9$ or $n-11$.

When $r=n-1$, $G$ is a tree and $mvd(G)=n$.

When $r=n-3$, combined with $2t\geq n+r-11$ and the fact that $r=t$ if and only if $G$ is a tree (i.e., $r=n-1$), $n-7\leq t\leq n-4$. Since $G$ is triangle-free, when $t=n-5, n-6, n-7$, respectively, $\sum\limits_{i=1}^{r} |B_{i}|\geq 8+2(n-5), 12+2(n-6), 16+2(n-7)$, contradicting the fact that $n=(\sum\limits_{i=1}^{r} |B_{i}|)-r+1$. Then $G$ has exactly one nontrivial block. Note that in other cases, we first use this method to determine the number of nontrivial blocks in $G$. Since $n=(\sum\limits_{i=1}^{r} |B_{i}|)-r+1$, this nontrivial block is of order $4$, i.e., $C_{4}$. According to Table \ref{6-10}, Theorem \ref{blockrelation} and \ref{everyblock}, $mvd(G)=n-2$.

Similarly, when $r=n-5$, we have $n-8\leq t\leq n-6$ and there are at most two nontrivial blocks in $G$. Since $n=(\sum\limits_{i=1}^{r} |B_{i}|)-r+1$ and $G$ is triangle-free, if $t=n-6$, then this nontrivial block is of order $6$, i.e., one of $\{P(2,2), P(2,1,1), P(1,1,1,1)\}$; if $t=n-7$, then there are two nontrivial blocks $B_{i}$ and $ B_{j}$ with $|B_{i}|+|B_{j}|=8$, i.e., both $B_{i}$ and $ B_{j}$ are $C_{4}$. According to Table \ref{6-10}, Theorem \ref{blockrelation} and \ref{everyblock}, $mvd(G)=n-3$ when the subgraph induced by all nontrivial blocks of $G$ is $P(2,2)$, and $mvd(G)=n-4$ for the rest cases.

When $r=n-7$, $n-9\leq t\leq n-8$. Since $n=(\sum\limits_{i=1}^{r} |B_{i}|)-r+1$, if $t=n-8$, then the order of this nontrivial block is $8$; if $t=n-9$, then there are two nontrivial blocks $B_{i}$ and $ B_{j}$ with $|B_{i}|+|B_{j}|=10$. Since $G$ is triangle-free, the subgraph induced by all nontrivial blocks of $G$ is one of the graphs in Fig. \ref{8}. According to Table \ref{6-10}, Theorem \ref{blockrelation} and \ref{everyblock}, $mvd(G)=n-4$ when the subgraph induced by all nontrivial blocks of $G$ is $P(3,3)$, $mvd(G)=n-5$ when it is one of $\{\Large{\textcircled{\small {4}}}, \Large{\textcircled{\small {13}}}, P(4,1,1), P(3,2,1), P(3,1,1,1), P(2,2,2)\}$, and $mvd(G)=n-6$ for the rest cases.

When $r=n-9$, $t=n-10$. Since $n=(\sum\limits_{i=1}^{r} |B_{i}|)-r+1$, this nontrivial block is of order $10$, i.e., one of the graphs in Appendix A. $10$ VERTEX. According to Table \ref{6-10}, Theorem \ref{blockrelation} and \ref{everyblock}, $mvd(G)=n-5$ when the subgraph induced by all nontrivial blocks of $G$ is $P(4,4)$, and $mvd(G)<n-5$ for the rest cases.

When $r=n-11$, $n-11\leq t\leq n-12$, a contradiction.
\end{proof}
\section{Erd\H{o}s-Gallai-type problems}
\quad \quad In this section, we first study the following extremal problem and obtain Theorem \ref{size}. To solve this problem, we show some lemmas.

For integers $k$ and $n$ with $1\leq k\leq n$, what is the maximum possible size of a connected graph $G$ of order $n$ with $mvd(G)=k$?

\begin{lem}\label{add}
For a connected graph $G$ of order $n$, the following holds:

(i) If $G$ is obtained by adding $k$ edges to $K_{n-1}$ from a vertex $v$ outside $K_{n-1}$, where $k\in [n-2]$, then $mvd(G)=n-k+1$.

(ii)If $G$ is obtained by removing any edge $e$ from $K_{n}$, where $n\geq3$, then $mvd(G)=3$; if $G$ is obtained by removing any two edges $e_1$ and $e_2$ from $K_{n}$, where $n\geq4$, then $mvd(G)\leq 4$.
\end{lem}
\begin{proof}
\textbf{(1)} Defined a vertex-coloring $\tau$ of $G$: $V(G)\rightarrow[n-k+1]$ such that $\tau(G-N(v))\rightarrow[n-k]$ and $\tau(N(v))=n-k+1$. If $x$ and $y$ are two nonadjacent vertices in $G$, then either $x$ or $y$ is $v$, say $v=x$. Since $\tau(N(v))=n-k+1$, $N(v)$ is a monochromatic $v$-$y$ vertex cut and $\tau$ is a $MVD$-coloring of $G$. So $mvd(G)\geq n-k+1$. On the contrary, let $u$ be a vertex that is nonadjacent to $v$. Then there are $k$ internally disjoint paths between $u$ and $v$. So, $mvd(G)\leq n-k+1$ by Theorem \ref{bound}.

\textbf{(2)} Let $v$ be one of the endpoints of $e$, then $G$ is obtained by adding $n-2$ edges from $v$ to $K_{n-1}$. From (1), $mvd(G)=3$.
If $e_1$ and $e_2$ are adjacent edges incident with a vertex $v$ of $G$, then $G$ is obtained by adding $n-3$ edges from $v$ to $K_{n-1}$. So $mvd(G)=4$. If $e_1$ and $e_2$ are nonadjacent, then $G=K_{2, 2}$ when $n=4$, or $G=K_{2, 2, 1, \ldots, 1}$ when $n>4$. According to Theorem \ref{mvd=1}(2), $mvd(G)=2$ or $mvd(G)=1$.
\end{proof}

\begin{lem}\label{k=2}
Let $G$ be a connected graph of order $n\geq 2$. Then the maximum size of $G$ with $mvd(G)=2$ is
\begin{equation}
    |E(G)|_{max}=
   \begin{cases}
   1,&\mbox{ $n=2$,}\\
   -\infty,&\mbox{ $n=3$,}\\
   4,&\mbox{ $n=4$,}\\
   7,&\mbox{ $n=5$,}\\
   \frac{n(n-1)}{2}-4,&\mbox{ $n\geq 6$.}\nonumber
   \end{cases}
\end{equation}
\end{lem}
\begin{proof}
\begin{figure}
\centering
  \includegraphics[height=4.5cm]{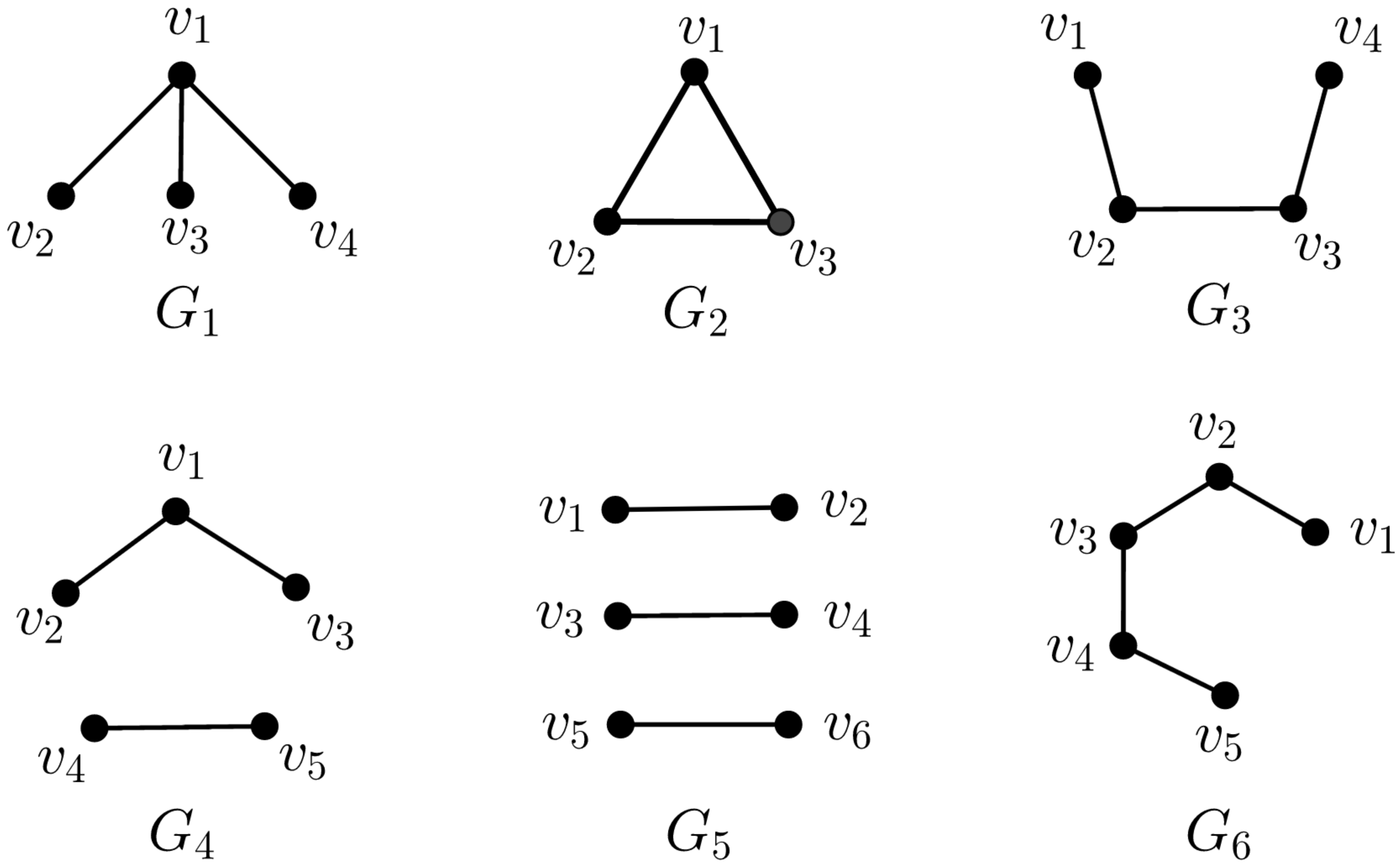}
  \caption{Proof process of Lemma \ref{k=2}}
  \label{5}
\end{figure}
Since $mvd(K_n)=n$, the maximum size of $G$ with $mvd(G)=n=2$ is $1$. There is no graph $G$ with $n=3$ and $mvd(G)=2$. According to Lemma \ref{add}(2), if $n\geq 4$ and $mvd(G)=2$, then $|E(G)|\leq |E(K_{n})|-2$ and the equation holds only when $n=4$. Now, let $n\geq 5$.

\textbf{Claim 1}: If $G$ is a graph of order $n$ obtained by removing any three edges from $K_{n}$, where $n\geq5 $, then $mvd(G)=2$ if and only if $n=5$ and the three removed edges are shown as $G_4$ in Fig. \ref{5}.

There are five cases to consider for the three removed edges (see $G_i, i\in [5]$ in Fig. \ref{5}), and if $n=5$, $G_5$ is excluded. Note that for case $G_i$, $i\in [4]$, $G$ is connected since $n\geq5$, and for case $G_5$, $G$ is connected only when $n>5$. For case $G_1$, $G$ is obtained by adding $n-4$ edges from $v_1$ to $K_{n-1}$. It follows by Lemma \ref{add}(1) that $mvd(G)=5$.
For the case of $G_2$, define a vertex-coloring $\tau: V(G)\rightarrow[4]$ such that $\tau(N(v_1))=\tau(N(v_2))=\tau(N(v_3))=1$, $\tau(v_1)=2$, $\tau(v_2)=3$ and $\tau(v_3)=4$. For any two nonadjacent vertices $x$ and $y$, $N(v_1)$ is a monochromatic $x$-$y$ vertex cut since $N(v_1)=N(v_2)=N(v_3)$. Then $mvd(G)\geq4$.
For the case of $G_3$, define a vertex-coloring $\tau: V(G)\rightarrow[3]$ such that $\tau(N(v_2))=\tau(N(v_3))=1$, $\tau(v_2)=2$ and $\tau(v_3)=3$. For any two nonadjacent vertices $x$ and $y$, $N(v_2)$ or $N(v_3)$ is a monochromatic $x$-$y$ vertex cut, then $mvd(G)\geq3$.

For the case of $G_4$, we claim that if $n=5$, then $mvd(G)=2$; if $n\geq6$, then $mvd(G)=1$. If $n=5$, define a vertex-coloring $\tau: V(G)\rightarrow[2]$ such that $\tau(v_1)=\tau(v_2)=\tau(v_3)=1$ and $\tau(v_4)=\tau(v_5)=2$. For any two nonadjacent vertices $x$ and $y$, if $\{x, y\}=\{v_4, v_5\}$, then $\{v_1, v_2, v_3\}$ is a monochromatic $x$-$y$ vertex cut; otherwise, $\{v_4, v_5\}$ is a monochromatic $x$-$y$ vertex cut. Then $mvd(G)\geq2$. Suppose that $mvd(G)>2$ and $\tau'$ is an $mvd$-coloring of $G$. For nonadjacent vertices $v_4$ and $v_5$, there must be $\tau'(v_1)=\tau'(v_2)=\tau'(v_3)$ since $N(v_4)=N(v_5)=\{v_1, v_2, v_3\}$. Thus $\tau'(v_4)\neq \tau'(v_5)$ and both are different from $\tau'(v_1)$, which contradicts the existence of monochromatic $v_1$-$v_2$ vertex cut since $N(v_1)\cap N(v_2)=\{v_4, v_5\}$. If $n\geq 6$ and $\tau'$ is an $mvd$-coloring of $G$, then $\tau'(V(G)-\{v_4, v_5\})$ is monochromatic since $N(v_4)=N(v_5)=V(G)-\{v_4, v_5\}$. For nonadjacent vertices $v_1$ and $v_2$, since $N(v_1)\cap N(v_2)=N(v_1)=V(G)-\{v_1, v_2, v_3\}$ and $(V(G)-\{v_4, v_5\}) \cap(V(G)-\{v_1, v_2, v_3\}) \neq \varnothing$ for $n\geq 6$, $\tau'(V(G)-\{v_1, v_2, v_3\})=\tau'(V(G)-\{v_4, v_5\})$. Therefore, $mvd(G)=1$.
At last, $G=K_{2, 2, 2, 1, \ldots, 1}$ for the case of $G_5$, then it follows by Theorem \ref{mvd=1}(2) that $mvd(G)=1$.

Therefore, the maximum size of $G$ with $n=5$ and $mvd(G)=2$ is $7$. Furthermore, if $n\geq 6$ and $mvd(G)=2$, then $|E(G)|\leq |E(K_{n})|-4$. Thus we only need to prove the following claim to complete the proof of Lemma \ref{k=2}

\textbf{Claim 2}: $G$ is a graph of order $n$, $n\geq 6 $, and if $G$ is obtained by removing any four edges from $K_{n}$, where the four removed edges are shown as $G_6$ in Fig. \ref{5}, then $mvd(G)=2$.

$G$ is connected since $n\geq 6$. Define a vertex-coloring $\tau: V(G)\rightarrow[2]$ such that $\tau(V(G)-\{v_3\})=1$ and $\tau (v_3)=2$. For any two nonadjacent vertices $x$ and $y$, $V(G)-\{x, y, v_3\}$ is a monochromatic $x$-$y$ vertex cut. Then $mvd(G)\geq 2$. Let $\tau'$ be an $mvd$-coloring of $G$. For nonadjacent vertices $v_4$ and $v_5$, $\tau'(V(G)-\{v_3, v_4, v_5\})$ is monochromatic since $N(v_4)\cap N(v_5)=N(v_4)=V(G)-\{v_3, v_4, v_5\}$. For nonadjacent vertices $v_1$ and $v_2$, since $N(v_1)\cap N(v_2)=N(v_2)=V(G)-\{v_1, v_2, v_3\}$ and $(V(G)-\{v_3, v_4, v_5\}) \cap (V(G)-\{v_1, v_2, v_3\}) \neq \varnothing$ for $n\geq 6$, $\tau'(V(G)-\{v_1, v_2, v_3\})=1$. Uncolored vertex $v_3$ adds a new color at most. Therefore, $mvd(G)=2$.
\end{proof}

\begin{proof}[Proof of Theorem \ref{size}]
If $mvd(G)=n$, then the size is maximum when $G=K_n$. By Lemma \ref{add}(2), if $G$ is a graph of order $n$ obtained by removing any edge from $K_{n}$, where $n\geq3$, then $mvd(G)=3$. So, if $mvd(G)=3$ and $n\geq 4$, then the maximum size of $G$ of order $n$ is $\frac{n(n-1)}{2}-1$. There is no graph $G$ with $1< n\leq 4$ and $mvd(G)=1$. By Lemma \ref{add}(2), if $mvd(G)=1$ and $n\geq5$, then the maximum size of $G$ is $\frac{n(n-1)}{2}-2$. If $mvd(G)=2$, we refer to Lemma \ref{k=2}. Now we consider the case $4\leq k\leq n-1$. Let $t$ denote the number of vertices with degree $n-1$. We first claim that if the size of a connected graph $G$ of order $n$ is at least $\frac{n(n-1)}{2}-k+3$, then $mvd(G)\leq k-1$.\\
\textbf{Case 1}. $n-k+2\leq t\leq n-1$.

For any two nonadjacent vertices $x$ and $y$ in $G$, since $x,\ y$ are adjacent to each vertex of degree $n-1$ in $G$, $|N(x)\cap N(y)|\geq n-k+2$. Therefore, $mvd(G)\leq k-1$ by Theorem \ref{bound}.\\
\textbf{Case 2}. $0\leq t\leq n-k+2$.

\textbf{Claim 1}: There is at least one vertex with degree $n-2$.

Otherwise, except $t$ vertices with degree $n-1$, the maximum degree of the remaining vertices in $G$ is $n-3$ at most. Then the size is
\begin{equation}
\begin{aligned}
|E(G)|&\leq \frac{t(n-1)+(n-t)(n-3)}{2}\\
&=\frac{n^{2}-3n+2t}{2}\\
&\leq \frac{n^{2}-3n+2(n-k+2)}{2}\\
&<\frac{n(n-1)}{2}-k+3,\nonumber
\end{aligned}
\end{equation}
a contradiction. Thus, $G$ has at least one vertex with degree $n-2$.

\textbf{Claim 2}: $\delta(G)\geq n-k+2$.

Otherwise, there is a vertex with degree $\leq n-k+1$, then the size is
\begin{equation}
\begin{aligned}
|E(G)|&\leq \frac{(n-k+2)(n-1)+(n-k+1)+(k-3)(n-2)}{2}\\
&=\frac{n^{2}-n-2k+5}{2}\\
&<\frac{n(n-1)}{2}-k+3,\nonumber
\end{aligned}
\end{equation}
a contradiction. Thus, $\delta(G)\geq n-k+2$.

Let $x$ be a vertex of $G$ with degree $n-2$. There is a vertex $y$ in $G$ which is nonadjacent to $x$. By Claim 2, $d(y)\geq n-k+2$, then $|N(x)\cap N(y)|\geq n-k+2$. In other words, there are at least $n-k+2$ internal disjoint $x,y$-paths. Therefore, $mvd(G)\leq k-1$ by Theorem \ref{bound}.

Above all, if $mvd(G)\geq k$, then the size of $G$ of order $n$ is at most $\frac{n(n-1)}{2}-k+2$. It remains to show that for any integers $k$ and $n$, where $4\leq k\leq n-1$, there must be a connected graph $G$ of order $n$ and size $\frac{n(n-1)}{2}-k+2$ such that $mvd(G)=k$. Suppose that $G$ is a graph of order $n$ and size $\frac{n(n-1)}{2}-k+2$ obtained by adding $n-k+1$ edges to $K_{n-1}$ from a vertex $v$ outside $K_{n-1}$. By Lemma \ref{add}(1), $mvd(G)=k$.
\end{proof}

\begin{proof}[Proof of Theorem \ref{E-G}]
It is worth mentioning that the parameter $f_v(n, k)$ is equivalent to another parameter. Let $s_v(n, k)=max\{|E(G)|:|G|=n, mvd(G)\leq k\}$. It is easy to see that $f_v(n, k)=s_v(n, k-1)+1$. Let $n\geq5$. There are three cases as follows.\\
\textbf{Case 1}. $k=1$.

Since $mvd(G)\geq1$ holds for any graph $G$ and the tree has minimum size, $f_v(n, 1)=n-1$.\\
\textbf{Case 2}. $2\leq k\leq3$.

According to Theorem \ref{size}, $|E(G)|_{max}=\frac{n(n-1)}{2}-2$ if $mvd(G)=1$, $|E(G)|_{max}=\frac{n(n-1)}{2}-3$ if $mvd(G)=2$ and $n=5$, and $|E(G)|_{max}=\frac{n(n-1)}{2}-4$ if $mvd(G)=2$ and $n\geq6$. So we have $s_v(n, 1)=s_v(n, 2)=\frac{n(n-1)}{2}-2$. Since $f_v(n, k)=s_v(n, k-1)+1$, $f_v(n, 2)=f_v(n, 3)=\frac{n(n-1)}{2}-1$.\\
\textbf{Case 3}. $4\leq k\leq n$.

According to Theorem \ref{size}, $|E(G)|_{max}=\frac{n(n-1)}{2}-1$ if $mvd(G)=3$, and $|E(G)|_{max}<\frac{n(n-1)}{2}-1$ if $3<mvd(G)\leq n-1$. So we have $s_v(n, 3)=s_v(n, 4)=\cdots=s_v(n, n-1)=\frac{n(n-1)}{2}-1$. Since $f_v(n, k)=s_v(n, k-1)+1$, $f_v(n, 4)=f_v(n, 5)=\cdots=f_v(n, n)=\frac{n(n-1)}{2}$.
\end{proof}

\begin{rem}\label{remark 5}
For positive integers $n,\ k$ with $1\leq k\leq n\leq 4$, the results are shown as follows.
\end{rem}
\begin{table}[h]
\begin{tabular}{|c|c|c|c|c|c|c|c|c|c|c|c|c|c|c|c|c|c|}
\hline
$n$  & \multicolumn{4}{c|}{1}& \multicolumn{4}{c|}{2}& \multicolumn{4}{c|}{3}& \multicolumn{4}{c|}{4} \\
\hline
$k$ & 1 & 2 & 3 & 4 & 1 & 2 & 3 & 4 & 1 & 2 & 3 & 4 & 1 & 2 & 3 & 4  \\
\hline
$|E(G)|_{max}$&0&-&-&-&-&1&-&-&-&-&3&-&-&4&5&6\\
\hline
$f_v(n,k)$&0&-&-&-&1&1&-&-&2&2&2&-&3&3&5&6\\
\hline
\end{tabular}
\end{table}

\section{Algorithm for $mvd$-coloring}
\quad \quad The monochromatic vertex-disconnection number of a graph comes from coloring by keeping a global property of a graph. In Theorem \ref{blockrelation} and \ref{everyblock}, we transformed the global property into a local property for each block, which greatly simplified the original problem.
Based on this, we propose an algorithm to obtain $mvd(G)$ and an $mvd$-coloring of $G$. For complex graphs, our algorithm can give an $mvd$-coloring quickly and accurately, or at least reduce the tedious work.
\subsection{Pseudo-code for $mvd$-coloring algorithm}
\quad \quad Our algorithm is based on the block decomposition algorithm, which proposed by Tarjan \cite{Tarjan} in 1972. However, Tarjan's algorithm is skillful and the part of this algorithm that deals with undirected graphs is less known. Thus, we first restate the block decomposition algorithm for undirected graphs in a more understandable way than in \cite{Tarjan}.

\makeatletter
\newenvironment{breakable1algorithm}
  {
   \begin{center}
     \refstepcounter{algorithm}
     \hrule height.8pt depth0pt \kern2pt
     \renewcommand{\caption}[2][\relax]{
       {\raggedright\textbf{\ALG@name~\thealgorithm} ##2\par}
       \ifx\relax##1\relax
         \addcontentsline{loa}{algorithm}{\protect\numberline{\thealgorithm}##2}
       \else
         \addcontentsline{loa}{algorithm}{\protect\numberline{\thealgorithm}##1}
       \fi
       \kern2pt\hrule\kern2pt
     }
  }{
     \kern2pt\hrule\relax
   \end{center}
  }
\makeatother

\begin{breakable1algorithm}\label{algorithm2}
\caption{Block decomposition algorithm for undirected graph $G$}
\hspace*{0.02in} {\bf Input:}
$(V(G),\ E(G))$, a root $x$ of $G$\\
\hspace*{0.02in} {\bf Output:}
set $CutSet$ of cut-vertices, set $BlocksSet$ of blocks
\begin{algorithmic}[1]
\State \textbf{for every} edge $vw\in E(G)$ mark $vw$ ``unexplored''
\For{\textbf{every} vertex $w \in V(G)$}
    \State $K(w) \gets 0$
    \State $f(w)$ $\gets$ null
\EndFor
\State $CutSet \gets$ null
\State $BlocksSet \gets$ null
\State $v \gets x$
\State $K(x) \gets 1$
\State $L(x) \gets 1$
\State $i \gets 1$
\State vacate $S$
\State push $x$ into $S$
\While{$v$ has an unexplored incident edge or $f(v) \neq$ null}
    \If{$v$ has an unexplored incident edge $vw$}
        \State mark $vw$ ``explored''
        \If{$K(w)=0$ ($w$ is unexplored)}
            \State push $w$ into $S$
            \State $f(w) \gets v$
            \State $i=i+1$
            \State $K(w) \gets i$
            \State $L(w) \gets i$
            \State $v \gets w$
        \Else{ ($w$ is explored)}
            \State $L(v)$ $\gets$ min$\{L(v),K(w)\}$
        \EndIf
    \Else{ ($f(v)$ is defined)}
        \If{$L(v)\geq K(f(v))$}
            \If{$f(v)\neq x$ or $x$ has an unexplored incident edge}
                   \State add $f(v)$ to $CutSet$
            \EndIf
            \State pop vertices from $S$ down to and including $v$
            \State the set of popped vertices, with $f(v)$, is an element of $BlocksSet$
        \Else{ ($L(v)<K(f(v))$)}
            \State $L(f(v))$ $\gets$ min$\{L(f(v)),L(v)\}$
        \EndIf
        $v \gets f(v)$
    \EndIf
\EndWhile
\end{algorithmic}
\end{breakable1algorithm}

\begin{figure}[H]
\centering
  \includegraphics[height=6cm]{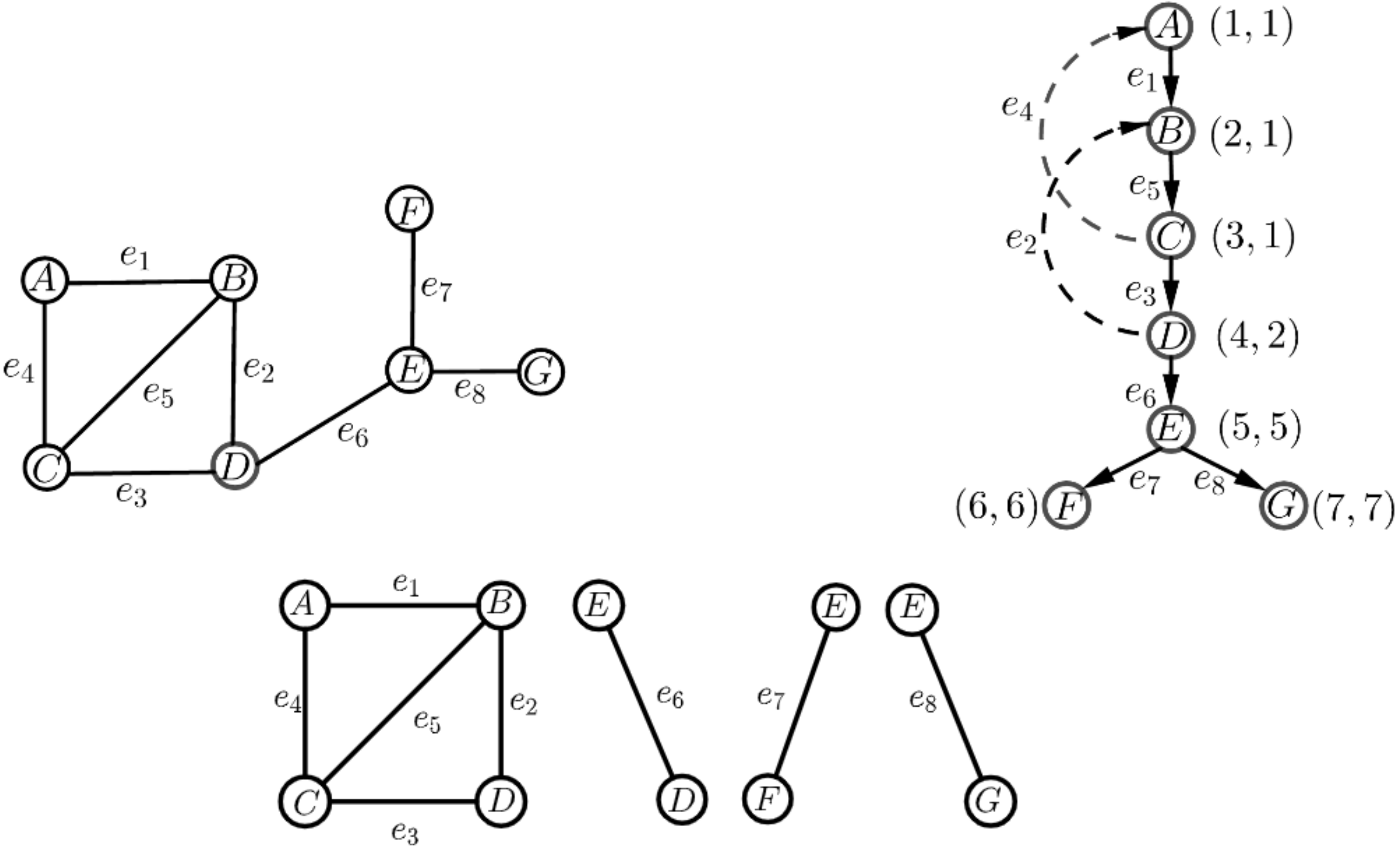}
  \caption{$(K(\cdot),\ L(\cdot))$, cut-vertices and the set of blocks}
  \label{1}
\end{figure}
Note that in Algorithm \ref{algorithm2}, we maintain the list of vertices to be searched as a stack (pushdown store). We store the vertices in the stack in the order of when they are explored. If on backtracking from $v$ to $f(v)$, we discover that $f(v)$ is a cut-vertex, we read and pop all vertices from the top of the stack down to and including $v$. All these vertices, plus $f(v)$, (which is not popped at this point from the stack) constitute a block.
Applying Algorithm \ref{algorithm2} to a graph, Fig. \ref{1} shows its $(K(\cdot),\ L(\cdot))$, cut-vertices and the set of blocks.

Algorithm \ref{algorithm2} produces the set of cut-vertices as well as the set of blocks. Theorem \ref{blockrelation} and \ref{everyblock} guarantee the correctness of Algorithm \ref{algorithm3}. We construct a $type\ set$ that consists of all graphs with known $mvd$-colorings. In line 4, we give block $Block$ an $mvd$-coloring (each time we color a new block with new colors), which can be achieved by finding the graph $thBlock$ that is isomorphic to $Block$ in type set, and coloring $Block$ according to $thBlock$'s $mvd$-coloring. The isomorphism algorithm is not the focus of this paper, so we do not elaborate on it. Since cut-vertices may be colored multiple times, in lines 7-17 we update the colors of the cut-vertices, as well as other related vertices, and finally give an $mvd$-coloring of $G$.

\makeatletter
\newenvironment{breakable2algorithm}
  {
   \begin{center}
     \refstepcounter{algorithm}
     \hrule height.8pt depth0pt \kern2pt
     \renewcommand{\caption}[2][\relax]{
       {\raggedright\textbf{\ALG@name~\thealgorithm} ##2\par}
       \ifx\relax##1\relax
         \addcontentsline{loa}{algorithm}{\protect\numberline{\thealgorithm}##2}
       \else
         \addcontentsline{loa}{algorithm}{\protect\numberline{\thealgorithm}##1}
       \fi
       \kern2pt\hrule\kern2pt
     }
  }{
     \kern2pt\hrule\relax
   \end{center}
  }
\makeatother

\begin{breakable2algorithm}\label{algorithm3}
\caption{Compute $mvd(G)$ and give an $mvd$-coloring of $G$}
\hspace*{0.02in} {\bf Input:}
$CutSet$ and $BlocksSet$ from Algorithm \ref{algorithm2}\\
\hspace*{0.02in} {\bf Output:}
$mvd(G)$, $mvd$-coloring of $G$
\begin{algorithmic}[1]
\State $mvd(G) \gets 0$
\State $r \gets card(BlocksSet)$
\For{\textbf{every} block $Block \in BlocksSet$}
    \State find the block $thBlock$ isomorphic to $Block$ in type set
    \State color $Block$ with new colors according to $thBlock$'s $mvd$-coloring
    \State $mvd(G) \gets mvd(G)+mvd(Block)$
\EndFor
$mvd(G) \gets mvd(G)-r+1$
\For{\textbf{every} vertex $u \in CutSet$}
    \State color($u$) $\gets$ null
\EndFor
\For{\textbf{every} block $Block \in BlocksSet$}
    \For{\textbf{every} vertex $v \in Block$}
        \If{$v$ is a cut-vertex}
           \State find the cut-vertex $u$ corresponding to $v$ in the $CutSet$
           \If{color($u$)=null}
               \State color($u$) $\gets$ color($v$)
           \Else
               \State color($v$) $\gets$ color($u$)
               \State update the colors of all vertices in $Block$ according to color($v$)
           \EndIf
        \EndIf
    \EndFor
\EndFor
\end{algorithmic}
\end{breakable2algorithm}

The main part of the code is shown in Appendix $B$ and the complete code is written in Java and given on Github: https://github.com/fumiaoT/mvd-coloring.git.
Our algorithm can give an $mvd$-coloring of $G$ quickly and accurately and reduce a lot of tedious work, provided that for any block of $G$, there is a graph in the type set that is isomorphic to it; otherwise, the program is also able to filter out blocks that are not isomorphic to any graph in the type set, and the $mvd$-coloring of these blocks, once solved, can further enrich our type set.

In the future, we can make an interactive graphical interface where the user can draw graphs on their web pages and our algorithm will get all the cut-vertices and blocks based on the graphs. Also the blocks are compared with the graphs in type set and $mvd$-colorings are performed. Finally, the cut-vertices, the colored graphs are labeled on the user's graphs. This algorithm is convenient and applicable as a mathematical tool to solve the $mvd$-coloring problem. Moreover, the algorithm can be modified in a similar way to solve a series of coloring problems, such as $rd$-coloring, $md$-coloring, etc.
\subsection{ An algorithm example }
\quad \quad The minimally 2-connected graphs with small order are elements of type set and are stored at $resource$ $directory$.
Thus if the blocks of graph $G$ are minimally 2-connected graphs with small order, our algorithm can give $G$ an $mvd$-coloring.
As an example, we apply our algorithm to the graph $G$ shown in Fig. \ref{6}$(1)$, which is stored into the computer by entering its adjacency matrix at the $main(\cdot)$ function in the $GraphContext\{\cdot\}$ class. And the output results are listed at the end of Appendix $B$.

The $DFSTarjan(\cdot)$ function in the $BlockAndCutVerticesBuilder\{\cdot\}$ class performs block decomposition on $G$, and outputs the cut-vertex $H$ and two blocks whose vertices and adjacency matrices are shown in $Block\ num\ 1$ and $Block\ num\ 2$, respectively (also Fig. \ref{7}$(1)(3)$).
The $getIsomorphicColors(\cdot)$ function in the $IsomorphicJudger\{\cdot\}$ class looks for graphs in type set that are isomorphic to the blocks of $G$, and outputs their correspondence at $Isomorphic$ $Relationship$ (also Fig. \ref{7}).
The $MvdColorMarker$ class performs $mvd$-coloring on $G$ and outputs the $mvd$-coloring of $G$ at $Coloring$ $Vertices$ $Results$ (also Fig. \ref{6}$(2)$).

\begin{figure}[H]
\centering
  \includegraphics[height=3.2cm]{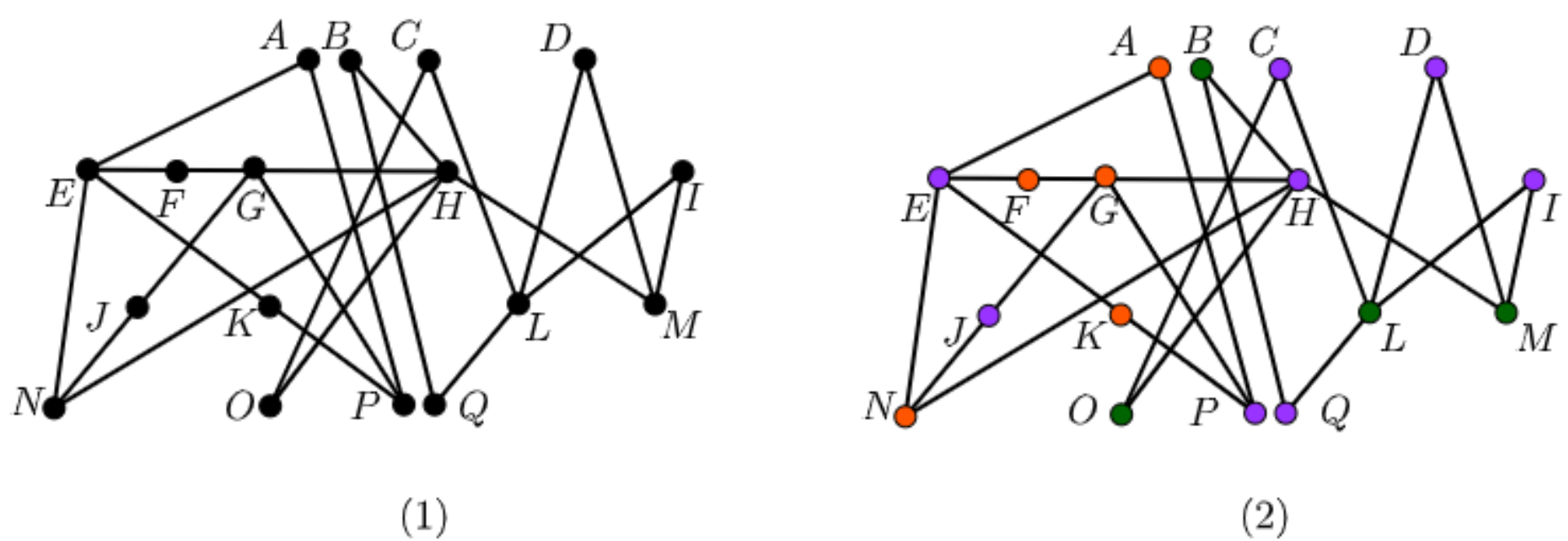}
  \caption{An example and its mvd-coloring}
  \label{6}
\end{figure}

\begin{figure}[H]
\centering
  \includegraphics[height=3.6cm]{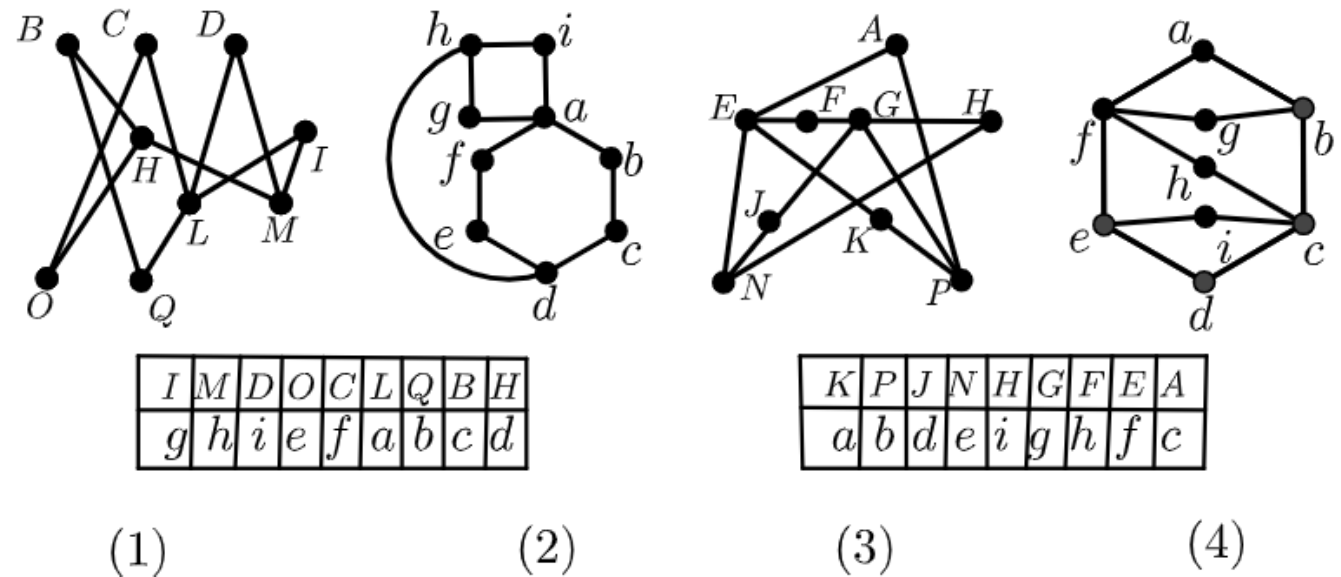}
  \caption{The output results of block decomposition and isomorphism}
  \label{7}
\end{figure}

\section*{Acknowledgement}

\quad \quad This work is supported by the National Natural Science Foundation of China (NSFC11921001), the Natural Key Research and Development Program of China (2018YFA0704701), the National Natural Science Foundation of China (NSFC11801410) and the National Natural Science Foundation of China (NSFC11971346). The authors are grateful to professor C. Zong for his supervision and discussion.

\newpage
\begin{appendix}
\section{Minimal blocks With 10 and Fewer Vertices}

\begin{figure}[H]
\centering
  \includegraphics[width=11.5cm]{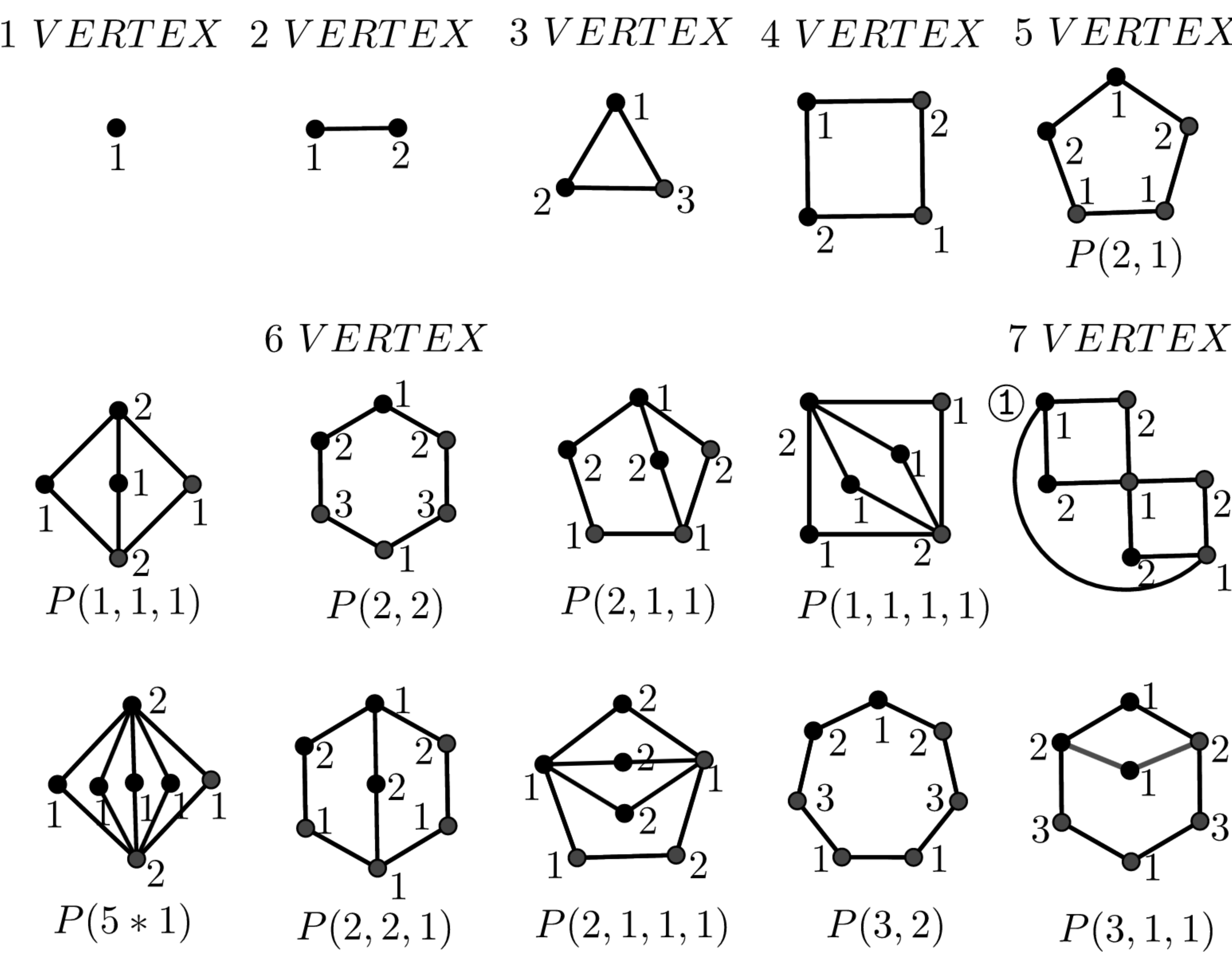}
  \includegraphics[width=12cm]{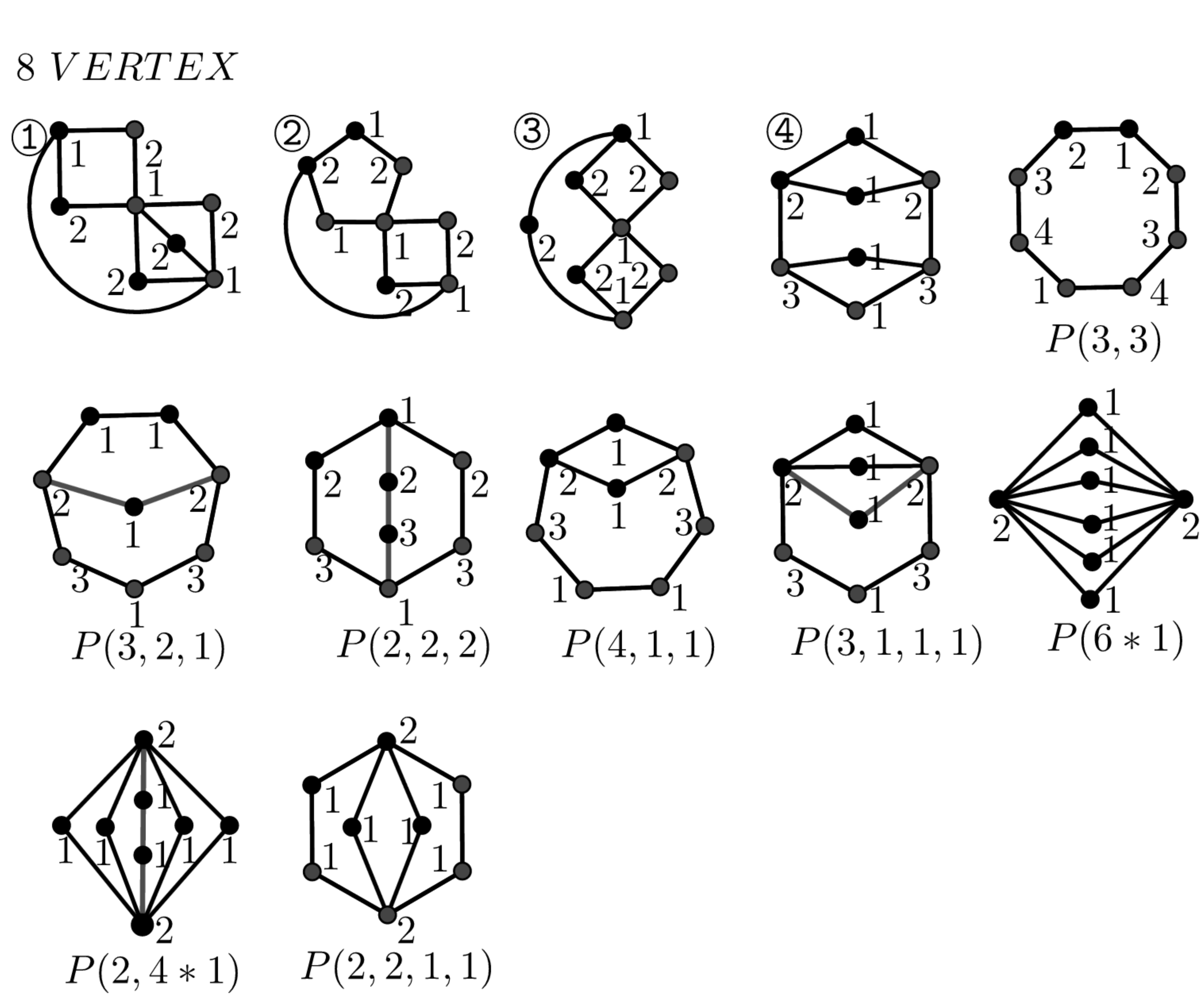}
\end{figure}
\begin{figure}[H]
\centering
  \includegraphics[height=8cm]{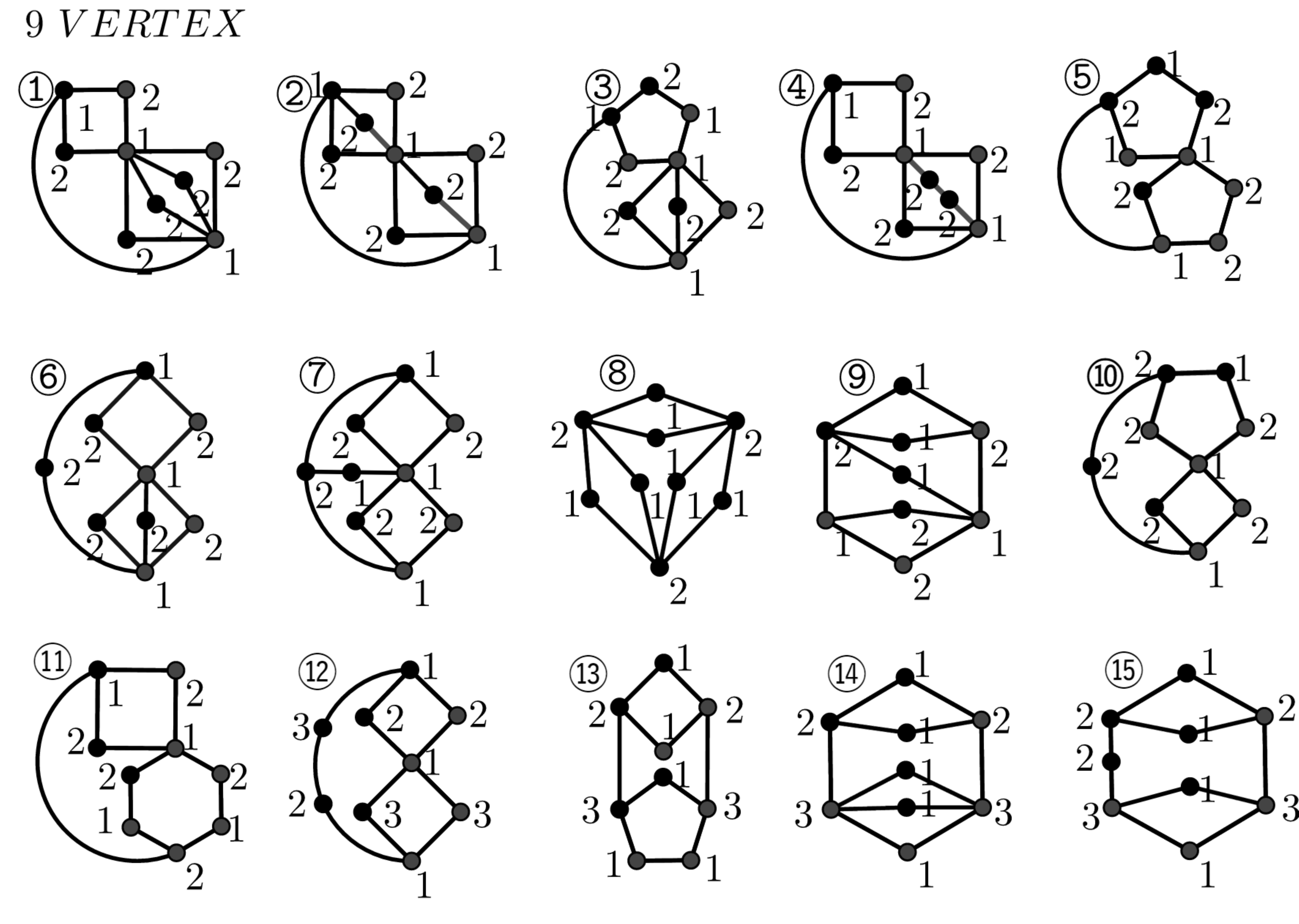}
  \includegraphics[height=10.5cm]{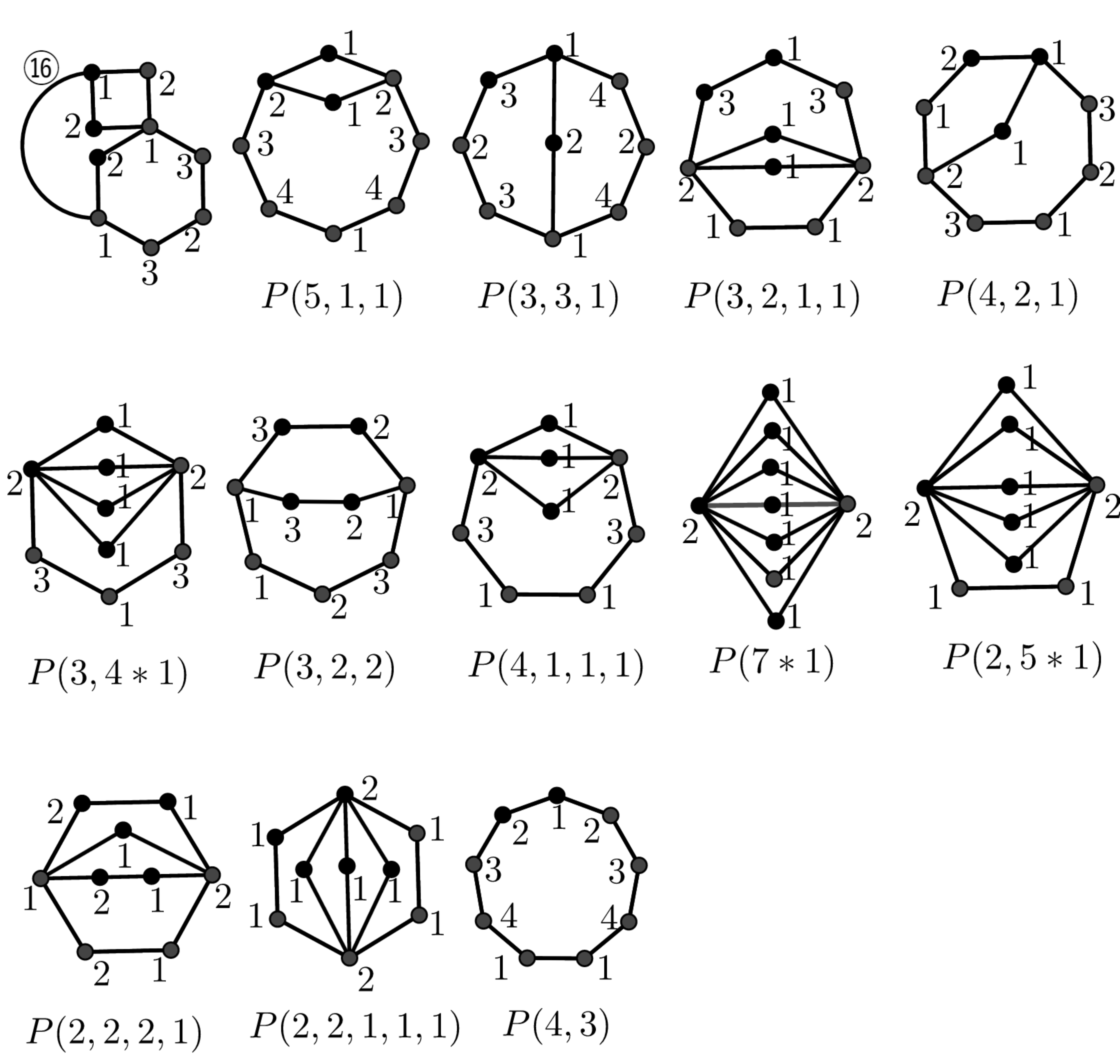}
\end{figure}
\begin{figure}[H]
\centering
  \includegraphics[height=10cm]{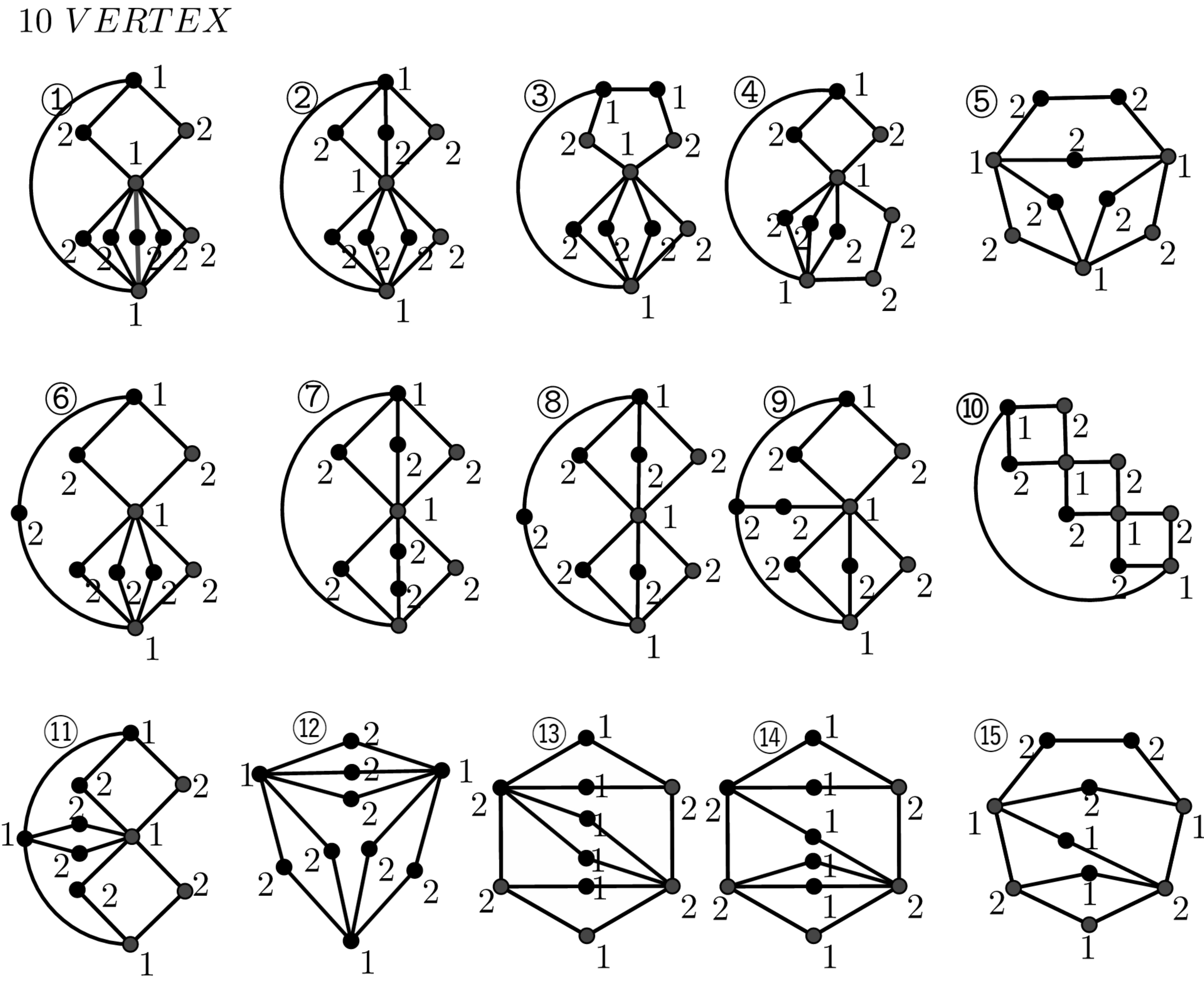}
  \includegraphics[height=7cm]{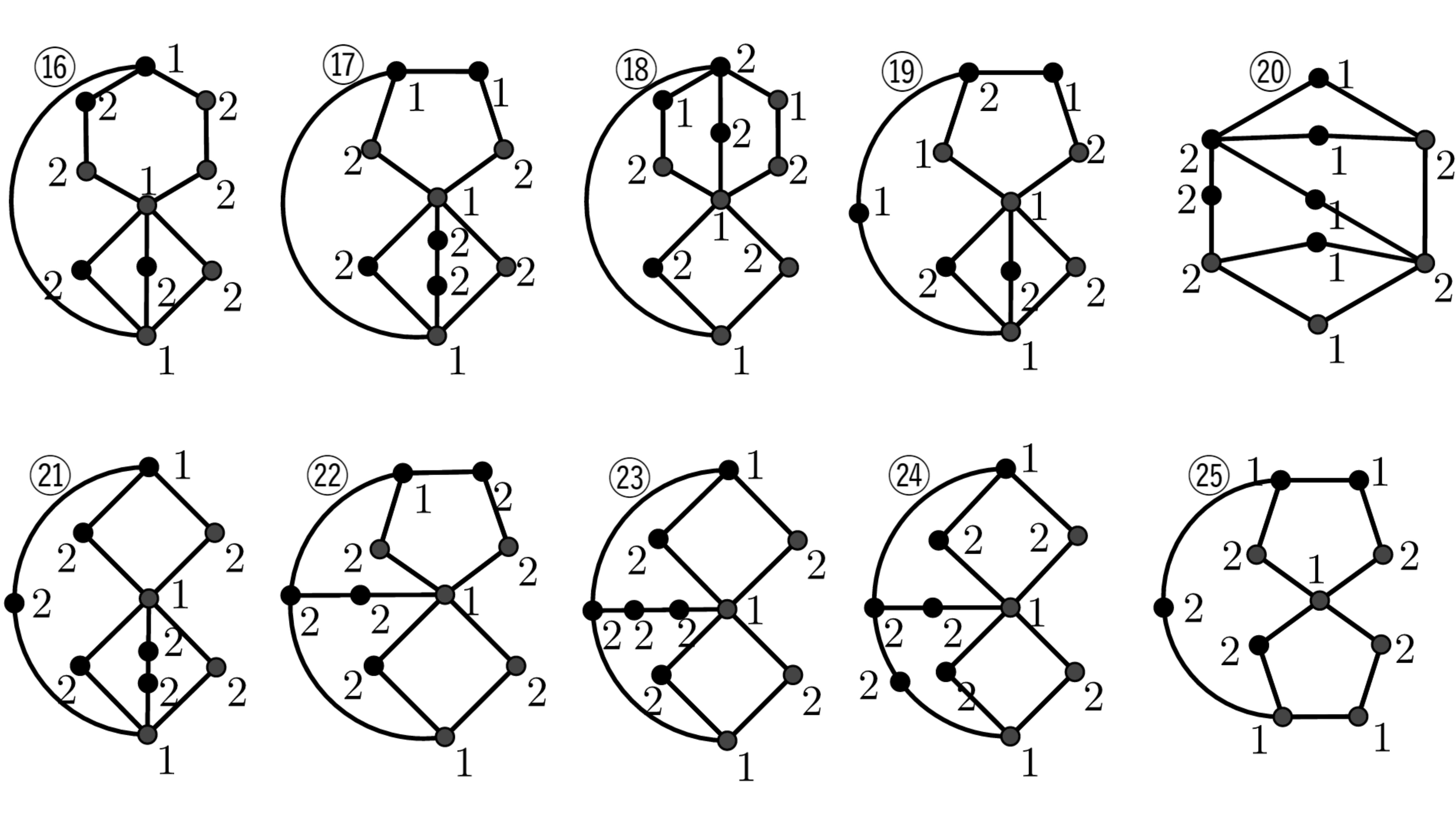}
  \includegraphics[height=3cm]{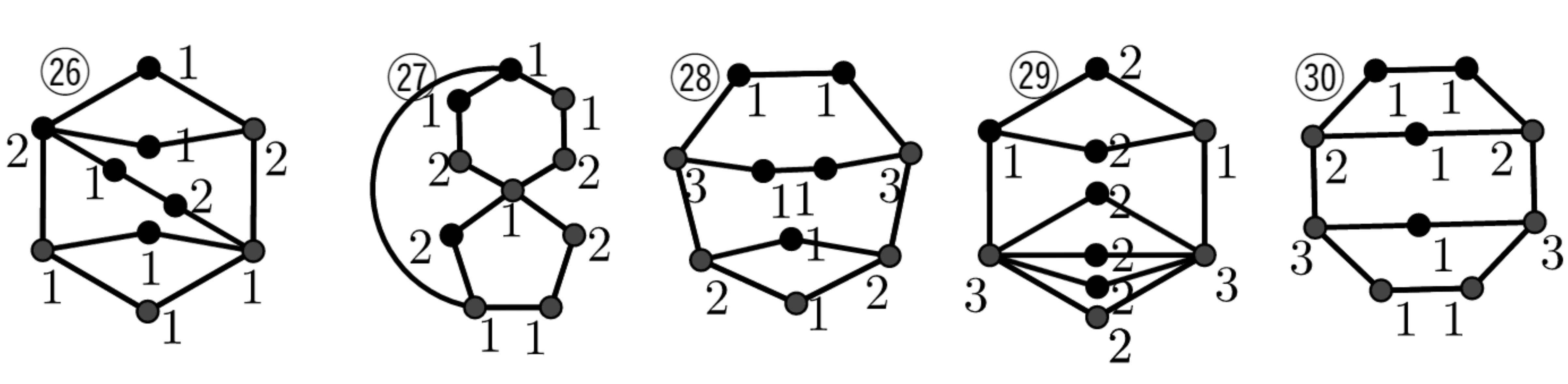}
\end{figure}
\begin{figure}[H]
\centering
  \includegraphics[height=6.2cm]{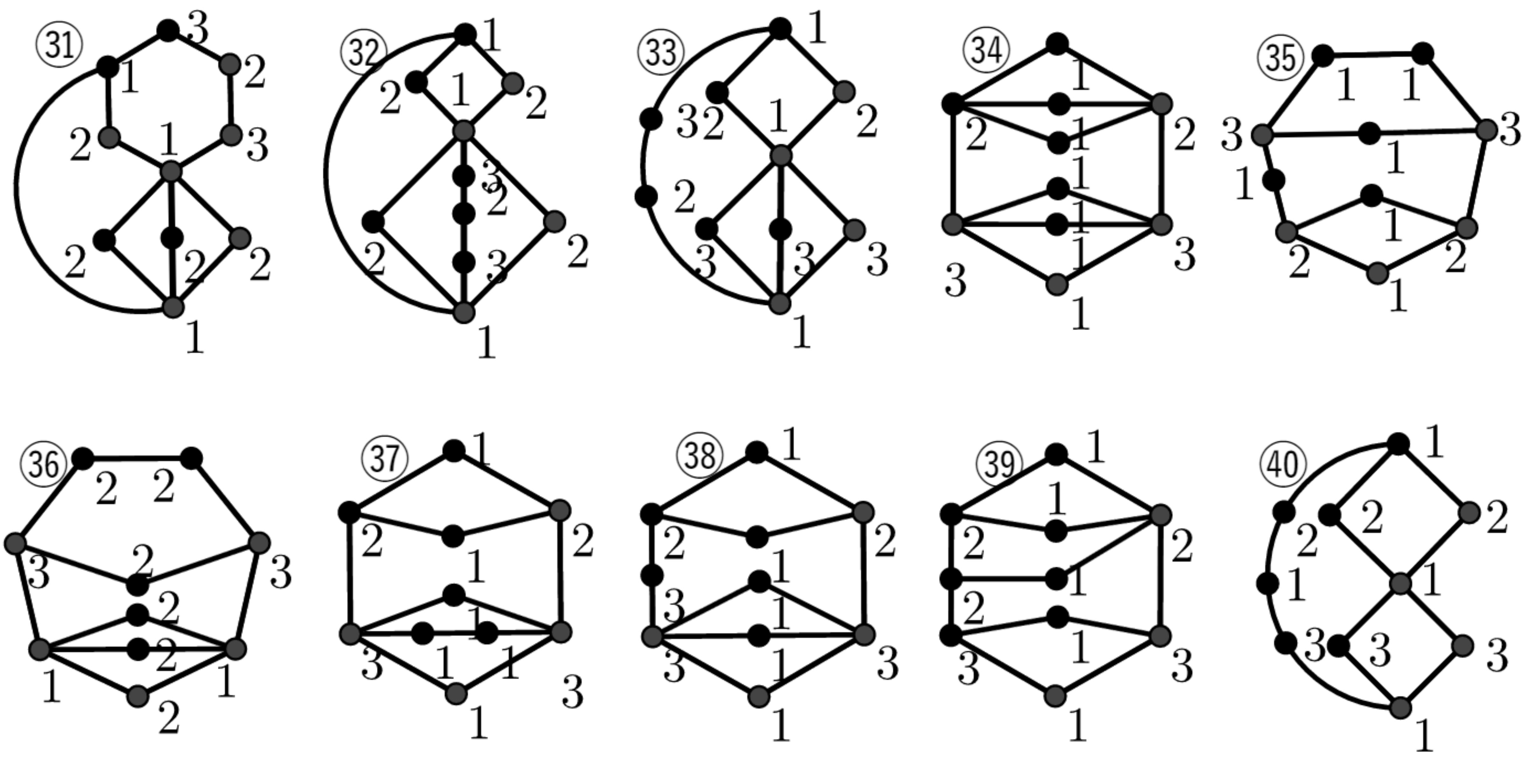}
  \includegraphics[height=7.2cm]{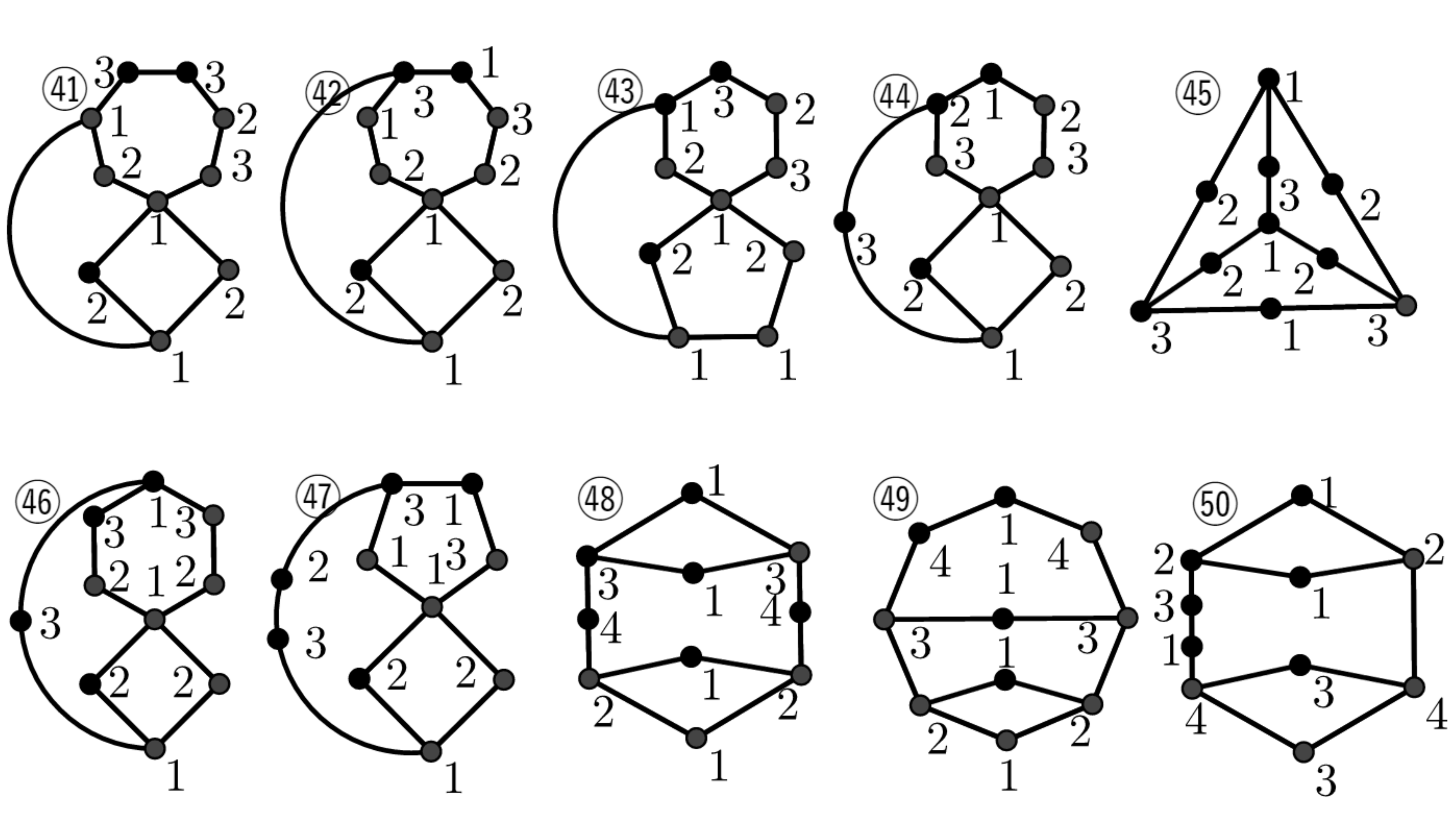}
  \includegraphics[height=7.9cm]{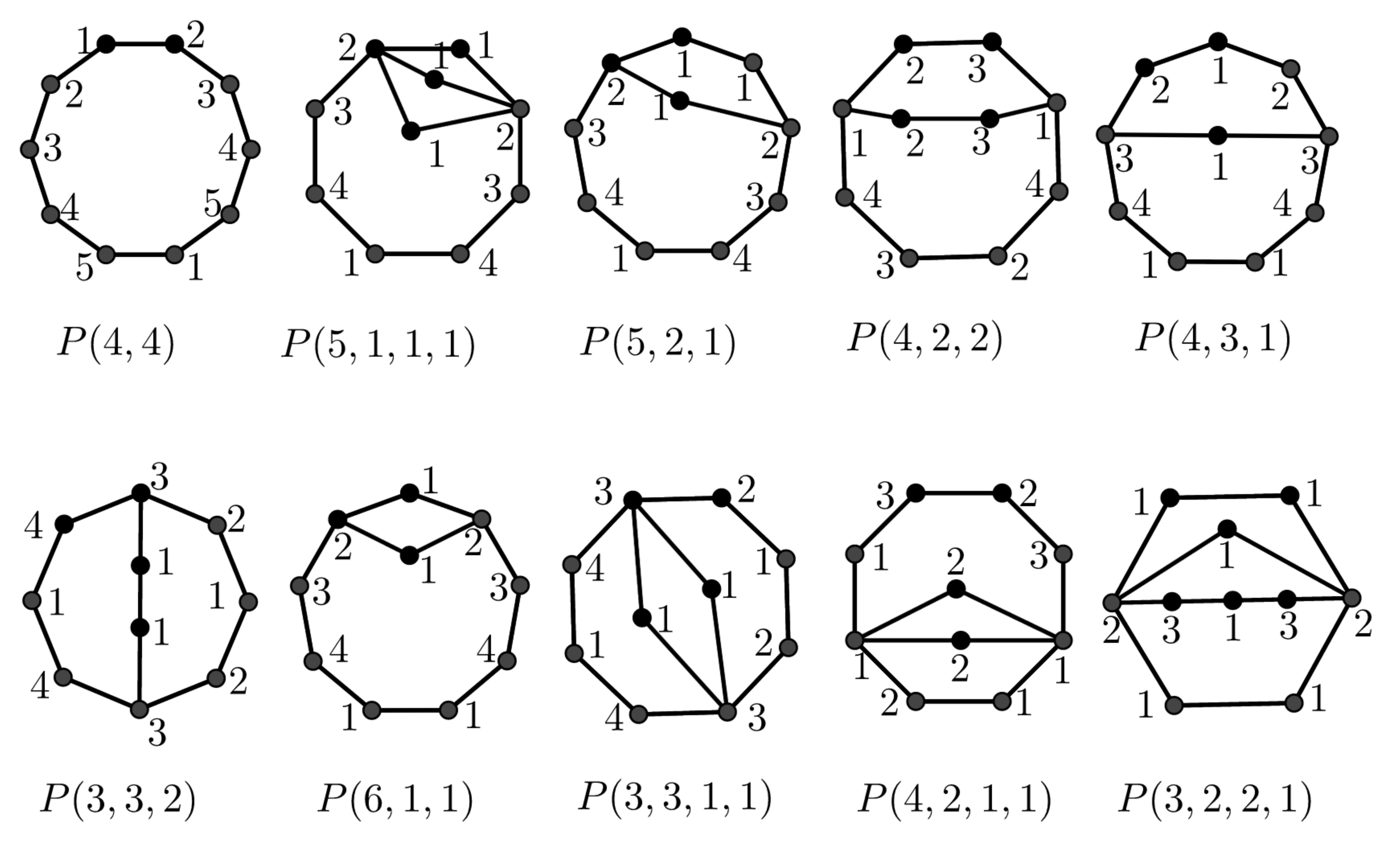}
\end{figure}
\begin{figure}[H]
\centering
  \includegraphics[height=8.4cm]{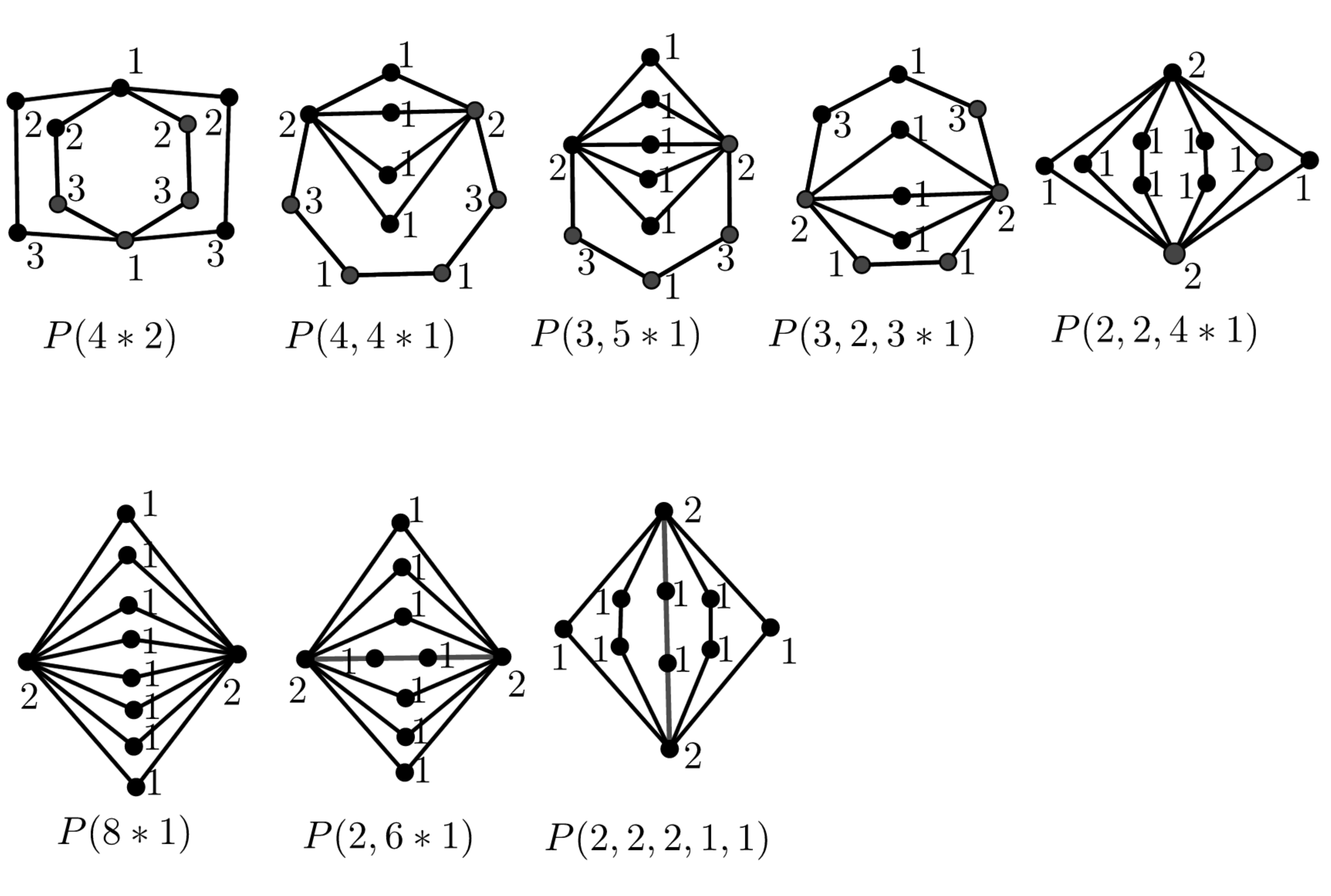}
\end{figure}

\newpage
\section{Code for $mvd$-coloring}

\begin{lstlisting}[ language=Java]
public class GraphContext {
    public static void main(String[] args) {
        String[] verx = new String[]{"A", "B", "C", "D", "E", "F", "G", "H", "I", "J", "K", "L", "M", "N", "O", "P", "Q"};
        int[][] edges = new int[][]{
                {0, 0, 0, 0, 1, 0, 0, 0, 0, 0, 0, 0, 0, 0, 0, 1, 0},
                {0, 0, 0, 0, 0, 0, 0, 1, 0, 0, 0, 0, 0, 0, 0, 0, 1},
                {0, 0, 0, 0, 0, 0, 0, 0, 0, 0, 0, 1, 0, 0, 1, 0, 0},
                {0, 0, 0, 0, 0, 0, 0, 0, 0, 0, 0, 1, 1, 0, 0, 0, 0},
                {1, 0, 0, 0, 0, 1, 0, 0, 0, 0, 1, 0, 0, 1, 0, 0, 0},
                {0, 0, 0, 0, 1, 0, 1, 0, 0, 0, 0, 0, 0, 0, 0, 0, 0},
                {0, 0, 0, 0, 0, 1, 0, 1, 0, 1, 0, 0, 0, 0, 0, 1, 0},
                {0, 1, 0, 0, 0, 0, 1, 0, 0, 0, 0, 0, 1, 1, 1, 0, 0},
                {0, 0, 0, 0, 0, 0, 0, 0, 0, 0, 0, 1, 1, 0, 0, 0, 0},
                {0, 0, 0, 0, 0, 0, 1, 0, 0, 0, 0, 0, 0, 1, 0, 0, 0},
                {0, 0, 0, 0, 1, 0, 0, 0, 0, 0, 0, 0, 0, 0, 0, 1, 0},
                {0, 0, 1, 1, 0, 0, 0, 0, 1, 0, 0, 0, 0, 0, 0, 0, 1},
                {0, 0, 0, 1, 0, 0, 0, 1, 1, 0, 0, 0, 0, 0, 0, 0, 0},
                {0, 0, 0, 0, 1, 0, 0, 1, 0, 1, 0, 0, 0, 0, 0, 0, 0},
                {0, 0, 1, 0, 0, 0, 0, 1, 0, 0, 0, 0, 0, 0, 0, 0, 0},
                {1, 0, 0, 0, 0, 0, 1, 0, 0, 0, 1, 0, 0, 0, 0, 0, 0},
                {0, 1, 0, 0, 0, 0, 0, 0, 0, 0, 0, 1, 0, 0, 0, 0, 0}
        };
        BlockAndCutVerticesBuilder blockAndCutVerticesBuilder = new BlockAndCutVerticesBuilder();
        Edge[][] matrixEdges = blockAndCutVerticesBuilder.create(verx, edges);
        System.out.println("### Input Graph ###");
        blockAndCutVerticesBuilder.printGraph(matrixEdges);
        //use tarjan algorithm to calculate cutVertices and blocks
        blockAndCutVerticesBuilder.makeDFSTarjan(0);
        //print cutVertices, blocks and Adjacency matrix
        blockAndCutVerticesBuilder.printResult();
        System.out.println();
        List<MvdGraph> graphs = blockAndCutVerticesBuilder.createMvdGraph();
        //mvd coloring
        MvdColorMarker mvdColorMarker = new MvdColorMarker();
        mvdColorMarker.markBlocks(graphs);
        Node[] vertices = blockAndCutVerticesBuilder.vertices;
        System.out.println("### Coloring Vertices Results ###");
        for (Node node : vertices) {
            System.out.print(node);
        }
    }
}
------------------------------------------------
public class BlockAndCutVerticesBuilder {
    public Node[] vertices;
    public Edge[][] edges;
    public Set<Node> cutVerticesSet = new HashSet<>();
    public List<List<Node>> block = new ArrayList<>();
    public List<Edge[][]> blocksEdges = new ArrayList<>();
    public Stack<Node> nodesStack = new Stack<>();
    int root;
    int count = 1;
    public Edge[][] create(String[] vertexs, int[][] edges) {
        this.vertices = new Node[vertexs.length];
        for (int i = 0; i < vertexs.length; i++) {
            this.vertices[i] = new Node(vertexs[i]);
        }
        this.edges = new Edge[edges.length][edges[0].length];
        for (int row = 0; row < edges.length; row++) {
            for (int col = 0; col < edges[0].length; col++) {
                this.edges[row][col] = new Edge(edges[row][col]);
            }
        }
        return this.edges;
    }
    public List<MvdGraph> createMvdGraph() {
        List<MvdGraph> graphs = new ArrayList<>();
        for (int i = 0; i < block.size(); i++) {
            graphs.add(new MvdGraph(block.get(i).toArray(new Node[0]), blocksEdges.get(i)));
        }
        return graphs;
    }
    public void printGraph(Edge[][] edges) {
        for (int row = 0; row < edges.length; row++) {
            for (int col = 0; col < edges[0].length; col++) {
                System.out.printf("%d\t", edges[row][col].connected);
            }
            System.out.println();
        }
    }
    public void printResult() {
        System.out.println("### CutVertices and Blocks  ###");
        System.out.println("cutVertices of Input Graph:");
        System.out.println("cutVerticesSet:" + cutVerticesSet);
        System.out.println();
        System.out.println("Block generated from Graph:");
        for (int i = 0; i < block.size(); i++) {
            System.out.println("Block num " + (i + 1));
            System.out.println(block.get(i));
            PrintUtils.printGraph(blocksEdges.get(i));
        }
    }
    public void makeDFSTarjan(int activeNodeIndex) {
        vertices[activeNodeIndex].depth = 1;
        nodesStack.push(vertices[activeNodeIndex]);
        root = activeNodeIndex;
        //calculate the cut vertices and blocks
        DFSTarjan(activeNodeIndex);
        //calculate the blocks' adjacency matrices
        cutBlocksEdges();
    }
    public void cutBlocksEdges() {
        for (List<Node> lists : block) {
            List<Integer> indices = new ArrayList<Integer>();
            for (Node node : lists) {
                indices.add(getIndexOfNode(node));
            }
            Edge[][] blockEdge = new Edge[indices.size()][indices.size()];
            for (int i = 0; i < indices.size(); i++) {
                for (int j = 0; j < indices.size(); j++) {
                    int rowIndex = indices.get(i);
                    int colIndex = indices.get(j);
                    blockEdge[i][j] = edges[rowIndex][colIndex];
                }
            }
            blocksEdges.add(blockEdge);
        }
    }
    private void DFSTarjan(int activeNodeIndex) {
        Node currentNode = this.vertices[activeNodeIndex];
        while(this.getNodesOfUnreachedEdge(activeNodeIndex) != null || currentNode.parent != null) {
            if (this.getNodesOfUnreachedEdge(activeNodeIndex) == null) {
                if (currentNode.low >= currentNode.parent.depth) {
                    if (this.getIndexOfNode(currentNode.parent) != this.root || this.getNodesOfUnreachedEdge(this.root) != null) {
                        this.cutVerticesSet.add(currentNode.parent);
                        currentNode.parent.isCutVertex = true;
                    }
                    ArrayList list = new ArrayList();
                    while(!currentNode.equals(this.nodesStack.peek())) {
                        list.add(this.nodesStack.pop());
                    }
                    list.add(this.nodesStack.pop());
                    list.add(currentNode.parent);
                    this.block.add(list);
                } else {
                    currentNode.parent.low = Math.min(currentNode.parent.low, currentNode.low);
                }
                return;
            }
            Node nextNode = this.getNodesOfUnreachedEdge(activeNodeIndex);
            this.markEdgeReached(currentNode, nextNode);
            if (nextNode.depth == 0) {
                this.nodesStack.push(nextNode);
                nextNode.parent = currentNode;
                ++this.count;
                nextNode.depth = this.count;
                nextNode.low = this.count;
                this.DFSTarjan(this.getIndexOfNode(nextNode));
            } else {
                currentNode.low = Math.min(currentNode.low, nextNode.depth);
            }
        }
    }
    private void markEdgeReached(Node currentNode, Node nextNode) {
        int from = this.getIndexOfNode(currentNode);
        int to = this.getIndexOfNode(nextNode);
        this.edges[from][to].reached = true;
        this.edges[to][from].reached = true;
    }
    private Node getNodesOfUnreachedEdge(int activeVex) {
        for(int col = 0; col < this.edges[activeVex].length; ++col) {
            Edge edge = this.edges[activeVex][col];
            if (edge.connected == 1) {
                Node cur = this.vertices[col];
                if (!edge.reached) {
                    return cur;
                }
            }
        }
        return null;
    }
    private int getIndexOfNode(Node node) {
        for(int i = 0; i < this.vertices.length; ++i) {
            if (this.vertices[i].name.equals(node.name)) {
                return i;
            }
        }
        return -1;
    }
}
------------------------------------------------
public class IsomorphicJudger {
    public static Node[] getIsomorphicColors(MvdGraph block, MvdGraph type) {
        if (block == null || type == null) {
            return null;
        }
        if (block.vertices == null || type.vertices == null) {
            return null;
        }
        if (block.vertices.length != type.vertices.length) {
            return null;
        }
        int n = block.vertices.length;
        //block adjacency matrix
        int[][] blockArray = new int[n][n];
        for (int i1 = 0; i1 < n; i1++) {
            blockArray[i1] = new int[n];
        }
        // block 01 adjacency matrix
        int[][] blockTongxin = new int[n][n];
        for (int i1 = 0; i1 < n; i1++) {
            blockTongxin[i1] = new int[n];
        }
        //block XOR matrix
        int[][] blockYihuo = new int[n][n];
        for (int i1 = 0; i1 < n; i1++) {
            blockYihuo[i1] = new int[n];
        }
        //block XNOR matrix
        int[][] blockTonghuo = new int[n][n];
        for (int i1 = 0; i1 < n; i1++) {
            blockTonghuo[i1] = new int[n];
        }
        //type adjacency matrix
        int[][] typeArray = new int[n][n];
        for (int i2 = 0; i2 < n; i2++) {
            typeArray[i2] = new int[n];
        }
        // type 01 adjacency matrix
        int[][] typeTongxin = new int[n][n];
        for (int i1 = 0; i1 < n; i1++) {
            typeTongxin[i1] = new int[n];
        }
        //block XOR matrix
        int[][] typeYihuo = new int[n][n];
        for (int i1 = 0; i1 < n; i1++) {
            typeYihuo[i1] = new int[n];
        }
        //block XNOR matrix
        int[][] typeTonghuo = new int[n][n];
        for (int i1 = 0; i1 < n; i1++) {
            typeTonghuo[i1] = new int[n];
        }
        for (int row = 0; row < block.edges.length; row++) {
            for (int col = 0; col < block.edges[0].length; col++) {
                blockArray[row][col] = block.edges[row][col].connected;
            }
        }
        for (int row = 0; row < type.edges.length; row++) {
            for (int col = 0; col < type.edges[0].length; col++) {
                typeArray[row][col] = type.edges[row][col].connected;
            }
        }
        oneZero(blockArray, blockTongxin, n);
        xor(blockArray, blockTongxin, blockYihuo, n);
        xnor(blockArray, blockTongxin, blockTonghuo, n);
        oneZero(typeArray, typeTongxin, n);
        xor(typeArray, typeTongxin, typeYihuo, n);
        xnor(typeArray, typeTongxin, typeTonghuo, n);
        return getTransVertices(blockArray, blockYihuo, blockTonghuo, typeArray, typeYihuo, typeTonghuo, n, type.vertices, 0, 0);
    }

                }
------------------------------------------------
public class MvdColorMarker {
    public static List<MvdGraph> GRAPHS = new ArrayList<>();
    public int colorCount = 0;
    //load the template coloring graphs in the files
    static {
        GRAPHS.add(ReadGraphUtils.readFiles("graph_9Vertex-9.txt"));
        GRAPHS.add(ReadGraphUtils.readFiles("graph_9Vertex-11.txt"));
    }
    public void markBlocks(List<MvdGraph> blocks) {
        for (MvdGraph block : blocks) {
            findTypeAndColoring(block);
        }
        //Change the color of those vertices in the template block that have the same color as the cut-vertex to the color of the current cut-vertex
        for (MvdGraph block : blocks) {
            updateColorsOfCutVertex(block);
        }
    }
    private void updateColorsOfCutVertex(MvdGraph block) {
        if (block.template == null || block.vertices == null) {
            return;
        }
        for (int i = 0; i < block.vertices.length; i++) {
            Node node = block.vertices[i];
            if (!node.isCutVertex) {
                continue;
            }
            Node templateNode = block.template[i];
            for (int j = 0; j < block.template.length; j++) {
                if (block.template[j].color.equals(templateNode.color)) {
                    block.vertices[j].color = node.color;
                }
            }
        }
    }
    private void findTypeAndColoring(MvdGraph block) {
        System.out.println("### Isomorphic Relationship ###");
        for (MvdGraph type : GRAPHS) {
            findTypeAndColoring(block, type);
        }
    }
    private void findTypeAndColoring(MvdGraph graph, MvdGraph type) {
        Node[] template = IsomorphicJudger.getIsomorphicColors(graph, type);
        if (template == null) {
            return;
        }
        for (int i = 0; i < graph.vertices.length; i++) {
            //coloring the vertices
            graph.vertices[i].color = template[i].color + colorCount;
        }
        colorCount += graph.vertices.length;
        graph.template = template;
        printRelationships(graph,type,template);
    }
    private void printRelationships(MvdGraph graph, MvdGraph type,Node[] template) {
        System.out.println("current block:");
        PrintUtils.printVerticesArray(graph.vertices);
        System.out.println("isomorphic block before elementary operations:");
        PrintUtils.printVerticesArray(type.vertices);
        PrintUtils.printGraph(type.edges);
        System.out.println("isomorphic block: after elementary operations:");
        PrintUtils.printVerticesArray(template);
        System.out.println();
    }
}
------------------------------------------------
public class Node {
    public String name;
    public Node parent;
    public int depth;
    public int low;
    public Integer color;
    public Boolean isCutVertex = false;
    public Node(String name) {
        this.name = name;
        this.depth = 0;
        this.low = this.depth;
    }
    public Node(String name, Integer color) {
        this.name = name;
        this.color = color;
        this.depth = 0;
        this.low = this.depth;
    }
    public String toString() {
        if (color == null) {
            return "{" + "'" + name + '\'' + '}' ;
        }
        return "{" + "'" + name + '\'' + ":" + color + '}' + " ";
    }
    public boolean equals(Object o) {
        if (this == o) {
            return true;
        }
        if (o == null || getClass() != o.getClass()) {
            return false;
        }
        Node node = (Node) o;
        return depth == node.depth && low == node.low && color == node.color && Objects.equals(name, node.name) && Objects.equals(parent, node.parent);
    }
    public int hashCode() {
        return Objects.hash(name);
    }
}
------------------------------------------------
public class Edge {
    public int connected;
    public boolean reached;
    public Edge(int connected) {
        this.connected = connected;
        this.reached = false;
    }
}
------------------------------------------------
public class MvdGraph {
    public Edge[][] edges;
    public Node[] vertices;
    //record the colored vertices corresponding to vertices
    public Node[] template;
    public MvdGraph(Node[] vertices, Edge[][] edges) {
        this.vertices = vertices;
        this.edges = edges;
    }
    public void print() {
        for (Node node : vertices) {
            System.out.print(node);
        }
        System.out.println();
        printGraph(edges);
    }
    private void printGraph(Edge[][] edges) {
        for (int row = 0; row < edges.length; row++) {
            for (int col = 0; col < edges[0].length; col++) {
                System.out.printf("%d\t", edges[row][col].connected);
            }
            System.out.println();
        }
        System.out.println("----------------");
    }
}
------------------------------------------------
public class PrintUtils {
    public static void printGraph(Edge[][] edges) {
        for (int row = 0; row < edges.length; row++) {
            for (int col = 0; col < edges[0].length; col++) {
                System.out.printf("%d\t", edges[row][col].connected);
            }
            System.out.println();
        }
    }
    public static void printVerticesArray(Node[] nodes) {
        for (Node node:nodes) {
            System.out.print(node);
        }
        System.out.println();
    }
}
------------------------------------------------
public class ReadGraphUtils {
    public static MvdGraph readFiles(String fileName) {
        try {
            FileReader in = new FileReader("resources/"+fileName);
            BufferedReader reader = new BufferedReader(in);
            String line, verticesLines;
            ArrayList<ArrayList<Integer>> edgesList = new ArrayList<ArrayList<Integer>>();
            verticesLines = reader.readLine();
            Node[] vertices = createVertices(verticesLines.split(","));
            while ((line = reader.readLine()) != null) {
                String[] str = line.split(",");
                ArrayList lineEdge = new ArrayList<>();
                for (String edge : str) {
                    lineEdge.add(Integer.parseInt(edge.trim()));
                }
                edgesList.add(lineEdge);
            }
            Edge[][] edges = createEdges(edgesList);
            if (edges.length != vertices.length) {
                System.out.println("The number of vertices and edges should be equivalent");
                return null;
            }
            return new MvdGraph(vertices, edges);
        } catch (Exception e) {
            System.out.println(e);
            return null;
        }
    }
    private static Node[] createVertices(String[] strVertices) {
        Node[] vertices = new Node[strVertices.length];
        for (int i = 0; i < strVertices.length; i++) {
            vertices[i] = new Node(strVertices[i].split(":")[0].trim(), Integer.parseInt(strVertices[i].split(":")[1].trim()));
        }
        return vertices;
    }
    private static Edge[][] createEdges(ArrayList<ArrayList<Integer>> edgesList) {
        Edge[][] result = new Edge[edgesList.size()][edgesList.size()];
        for (int row = 0; row < result.length; row++) {
            for (int col = 0; col < result[0].length; col++) {
                result[row][col] = new Edge(edgesList.get(row).get(col));
            }
        }
        return result;
    }
}
------------------------------------------------
//resources directory
a:1, b:2, c:1, d:2, e:1, f:2, g:1, h:1, i:2
0, 1, 0, 0, 0, 1, 0, 0, 0
1, 0, 1, 0, 0, 0, 1, 0, 0
0, 1, 0, 1, 0, 0, 0, 1, 1
0, 0, 1, 0, 1, 0, 0, 0, 0
0, 0, 0, 1, 0, 1, 0, 0, 1
1, 0, 0, 0, 1, 0, 1, 1, 0
0, 1, 0, 0, 0, 1, 0, 0, 0
0, 0, 1, 0, 0, 1, 0, 0, 0
0, 0, 1, 0, 1, 0, 0, 0, 0

a:1, b:2, c:1, d:2, e:1, f:2, g:2, h:1, i:2
0, 1, 0, 0, 0, 1, 1, 0, 1
1, 0, 1, 0, 0, 0, 0, 0, 0
0, 1, 0, 1, 0, 0, 0, 0, 0
0, 0, 1, 0, 1, 0, 0, 1, 0
0, 0, 0, 1, 0, 1, 0, 0, 0
1, 0, 0, 0, 1, 0, 0, 0, 0
1, 0, 0, 0, 0, 0, 0, 1, 0
0, 0, 0, 1, 0, 0, 1, 0, 1
1, 0, 0, 0, 0, 0, 0, 1, 0
------------------------------------------------
//the output results
### CutVertices and Blocks  ###
cutVerticesSet:[{'H'}]
Block generated from Graph:
Block num 1
[{'I'}, {'M'}, {'D'}, {'O'}, {'C'}, {'L'}, {'Q'}, {'B'}, {'H'}]
0	1	0	0	0	1	0	0	0	
1	0	1	0	0	0	0	0	1	
0	1	0	0	0	1	0	0	0	
0	0	0	0	1	0	0	0	1	
0	0	0	1	0	1	0	0	0	
1	0	1	0	1	0	1	0	0	
0	0	0	0	0	1	0	1	0	
0	0	0	0	0	0	1	0	1	
0	1	0	1	0	0	0	1	0	
Block num 2
[{'K'}, {'P'}, {'J'}, {'N'}, {'H'}, {'G'}, {'F'}, {'E'}, {'A'}]
0	1	0	0	0	0	0	1	0	
1	0	0	0	0	1	0	0	1	
0	0	0	1	0	1	0	0	0	
0	0	1	0	1	0	0	1	0	
0	0	0	1	0	1	0	0	0	
0	1	1	0	1	0	1	0	0	
0	0	0	0	0	1	0	1	0	
1	0	0	1	0	0	1	0	1	
0	1	0	0	0	0	0	1	0	
### Isomorphic Relationship ###
current block:
{'I':2} {'M':1} {'D':2} {'O':1} {'C':2} {'L':1} {'Q':2} {'B':1} {'H':2}
isomorphic block before elementary operations:
{'a':1} {'b':2} {'c':1} {'d':2} {'e':1} {'f':2} {'g':2} {'h':1} {'i':2}
0	1	0	0	0	1	1	0	1	
1	0	1	0	0	0	0	0	0	
0	1	0	1	0	0	0	0	0	
0	0	1	0	1	0	0	1	0	
0	0	0	1	0	1	0	0	0	
1	0	0	0	1	0	0	0	0	
1	0	0	0	0	0	0	1	0	
0	0	0	1	0	0	1	0	1	
1	0	0	0	0	0	0	1	0	
isomorphic block: after elementary operations:
{'g':2} {'h':1} {'i':2} {'e':1} {'f':2} {'a':1} {'b':2} {'c':1} {'d':2}
### Isomorphic Relationship ###
current block:
{'K':10} {'P':11} {'J':11} {'N':10} {'H':11} {'G':10} {'F':10} {'E':11} {'A':10}
isomorphic block before elementary operations:
{'a':1} {'b':2} {'c':1} {'d':2} {'e':1} {'f':2} {'g':1} {'h':1} {'i':2}
0	1	0	0	0	1	0	0	0	
1	0	1	0	0	0	1	0	0	
0	1	0	1	0	0	0	1	1	
0	0	1	0	1	0	0	0	0	
0	0	0	1	0	1	0	0	1	
1	0	0	0	1	0	1	1	0	
0	1	0	0	0	1	0	0	0	
0	0	1	0	0	1	0	0	0	
0	0	1	0	1	0	0	0	0	
isomorphic block: after elementary operations:
{'a':1} {'b':2} {'d':2} {'e':1} {'i':2} {'g':1} {'h':1} {'f':2} {'c':1}
### Coloring Vertices Results ###
{'A':10} {'B':1} {'C':11} {'D':11} {'E':11} {'F':10} {'G':10} {'H':11} {'I':11} {'J':11} {'K':10} {'L':1} {'M':1} {'N':10} {'O':1} {'P':11} {'Q':11}
\end{lstlisting}
\end{appendix}
\end{document}